%% file: DLQR_automatica_final.tex
\documentclass[twocolumn]{autart}
\usepackage{graphicx}     
\usepackage{amsmath}
\usepackage{amssymb}
\usepackage{mathrsfs}
\usepackage{ntheorem}
\usepackage{bbm}
\usepackage{bm}
\usepackage{amsfonts}
\usepackage{cases}
\usepackage{graphicx}
\usepackage{wrapfig}
\usepackage{float}
\usepackage{booktabs}
\usepackage{breakcites}
\usepackage{multirow}
\usepackage{xcolor}
\usepackage{algorithm}
\usepackage{algpseudocode}
\usepackage{graphicx}
\usepackage{subcaption}
\usepackage{epstopdf}
\usepackage{url}
\usepackage{balance}
\usepackage{natbib}
\usepackage{comment}
\usepackage{rotating}
\usepackage{wrapfig}
\graphicspath{{./}{Fig/}}

\allowdisplaybreaks[4]

\algrenewcommand\algorithmicrequire{\textbf{Input:}}
\algrenewcommand\algorithmicensure{\textbf{Output:}}

\newcommand{\Dstab}{\mathcal{D}_{\mathrm{stab}}}
\newcommand{\tr}{\operatorname{tr}}
\newcommand{\R}{\mathbb{R}}

\newcommand{\lmin}{\lambda_{\min}}

\newtheorem{lemma}{Lemma}
\newtheorem{corollary}{Corollary}
\newtheorem{proposition}{Proposition}
\newtheorem{theorem}{Theorem}
\newtheorem{remark}{Remark}
\newtheorem{example}{Example}
\newtheorem{assumption}{Assumption}

\theoremstyle{nonumberplain}
\newtheorem{proof}{Proof}
\begin{document}

\begin{frontmatter}
\title{Exact Variance of Random Return in Distributional LQR and Its Application  to Mean-Variance Optimal Control}

\author[Imperial]{Ruyi Teng}\ead{ruyi.teng24@imperial.ac.uk},
\author[KTH]{Zifan Wang}\ead{zifanw@kth.se},
\author[Imperial,cor]{Yulong Gao}\corauth[cor]{Corresponding author.}\ead{yulong.gao@imperial.ac.uk}

\address[Imperial]{Department of Electrical and Electronic Engineering,
Imperial College London, United Kingdom}
\address[KTH]{Division of Decision and Control, KTH Royal Institute of Technology}

\begin{abstract}
The classical linear quadratic regulator (LQR) minimizes the expected cumulative return but fails to account for performance variability, rendering it inadequate for risk-aware applications. 
To address this, we introduce the variance of the cumulative return as a risk measure in LQR.
We derive the first exact closed-form expression for the variance of the discounted infinite-horizon return within the discrete-time Distributional LQR framework, for i.i.d. disturbances with symmetric probability densities.
For Gaussian disturbances, this expression elegantly simplifies to a form dependent only on the disturbance covariance. Leveraging these theoretical foundations, we formulate a mean-variance optimal control problem that explicitly manages the trade-off between expected return and performance variability. To address the resulting non-convex optimization problem, we propose a novel adjoint gradient descent algorithm for a penalized formulation of the original problem, and establish that all iterates remain stabilizing and converge to a stationary point of the penalized objective.
The effectiveness of this framework and the inherent risk-performance trade-off are demonstrated through numerical experiments.

\end{abstract}

\end{frontmatter}

\thispagestyle{empty}
\pagestyle{plain}

\begin{sloppypar}
\frenchspacing
	\input{Sec/intro}

        \input{Sec/problem}

     \input{Sec/variance} \input{Sec/optimalcontrol}

        \input{Sec/casestudy} \input{Sec/conclusion}

      \input{Sec/Appendix_A}
      \input{Sec/Appendix_B}

      \input{Sec/Appendix_C}
    \input{Sec/Appendix_D}

      \input{Sec/Appendix_E}
    \bibliographystyle{agsm}
\footnotesize{\bibliography{reference}}

\end{sloppypar}
\end{document}

%% file: Sec/intro.tex
\section{Introduction}
The linear quadratic regulator (LQR)
is a classical optimal control problem which provides linear state feedback control strategies to guarantee optimal expected cumulative quadratic return. It is a cornerstone in modern control and has been widely utilized in various domains, such as biology \citep{li2004iterative}, motor control \citep{todorov2002optimal}, economics \citep{hansen1980formulating} and robotics \citep{perez2012lqr}.
Although the structure of the expected cumulative return is completely understood by the Algebraic Riccati equation, all real-world systems live with various uncertainties and randomness. The optimality of the classical LQR has been shown to be inadequate in the presence of randomness \citep{tzortzis2015dynamic}, as it focuses solely on the mean and does not account for higher-order statistical properties.
Among these properties, the variance of the cumulative quadratic return offers a natural extension to evaluate its variability, which enables us to consider risk-aware control criteria, e.g., mean-variance optimal control. For instance, in building climate control, a controller minimizing only the expectation may achieve good average energy consumption but exhibit high day-to-day variability, leading to occasional large energy bills. Explicitly minimizing the mean-variance trade-off is therefore both natural and practically important.
Note that solving risk-aware control problems is important for many other practical systems, e.g., autonomous driving~\citep{gao2021risk}, power systems \citep{10363432}, and portfolio management \citep{markowitz2008portfolio,sood2026deep}. To the best of our knowledge, an exact closed-form characterization of the return variance in the discrete-time LQR has not yet been established. This gap motivates us to consider the following fundamental questions:
\begin{center}
   \emph{How can the return variance induced by a given policy be evaluated analytically in discrete-time LQR, and how can this variance be optimized explicitly through feedback design?}
\end{center}
 Addressing these questions has two important implications. 
First, an analytical expression for the variance provides a more complete assessment of the cumulative return associated with a given state-feedback policy, beyond what is captured by the expected return alone. Second, such a characterization is essential for developing a tractable mean–variance optimal control framework that incorporates risk sensitivity directly into the controller synthesis process. From this perspective, variance is not merely a post-hoc statistical descriptor of performance, but a controllable quantity that can be shaped through appropriate feedback design.
\subsection{Contributions}
Building on the recent characterization of \emph{cumulative} random return in the distributional LQR~\citep{wang2023policy,11020758,teng2025quadratic}, this paper aims to study the closed-form variance expression of the distributional LQR and develop a mean-variance optimal control framework. 
\begin{enumerate}
    \item We derive an exact closed-form expression for the variance of the cumulative random return in a discrete-time setting under i.i.d. disturbances whose probability density functions are symmetric with respect to the origin (Theorem~\ref{Theorem:mainresult}). We establish an upper bound on this variance (Corollary~\ref{theorem:bound}), which is shown empirically to be tighter than the bound in~\cite{11020758}. 
    \item When disturbances are i.i.d., zero-mean Gaussian, we show that the exact variance is further refined based solely on the covariance of disturbances (Theorem~\ref{theorem:variance}).
    \item We formulate a mean–variance optimal control framework that can explicitly trade off control performance and risk. To address this non-convex optimization problem, we propose a novel adjoint gradient descent algorithm (Algorithm~\ref{Alg:Adjoinddescent}) for a penalized formulation of the original problem and establish its convergence guarantees (Theorem~\ref{thm:conv}). 
\end{enumerate}

\subsection{Related work}
Our work focuses on variance characterization and mean-variance optimal control in the LQR setting.
Since variance is  a  second-moment risk descriptor, this positions our work within risk-aware LQR problems, which mainly focus on replacing the original expectation objective by risk measures to analyze the statistical properties of the involved random variables.  A core question for these problems is whether risk-aware criteria are computationally tractable. In \citep{jacobson1973optimal,whittle1990risk},  linear exponential quadratic  Gaussian (LEQG) is developed to penalize tail behavior using exponential functions of quadratic forms, which yields an analytical solution to the optimal control problem. The direct integration of the LQR with the Conditional Value-at-Risk (CVaR) is not computationally tractable, since CVaR rarely enjoys a closed-form
expression  and is computationally expensive.  In \cite{chapman2021toward}, the upper bound on the CVaR-defined finite cost against distributional ambiguity is developed for tractable sub-optimal control design. In \cite{tsiamis2020risk}, a
risk-constrained LQR is  proposed to consider the cumulative expected
predictive variance of the state cost as a risk measure, not the variance of discounted infinite-horizon cumulative cost in our work. 
On the other hand, distributional robustness is another view of introducing risk into LQ control problems. In \cite{van2015distributionally}, a distributionally robust chance-constrained LQ control is proposed using only mean/covariance information. In addition, \cite{kim2023distributional} address a min-max LQ control problem by synthesizing a control policy by adopting an adversary to select the worst-case distribution at each time step.  However, by optimizing against the worst-case distribution of uncertain disturbances, these distributionally robust approaches can yield overly conservative control policies, as the worst-case scenario may be highly improbable in practice.

Another related concept is the  mean-variance structure and its application in risk-averse circumstances. Originally in portfolio theory, risk is modeled by the second moment and penalized by means of a risk aversion parameter (\cite{markowitz2008portfolio}). Most mean–variance studies focus on portfolio selection \cite{rubinstein1973mean,kim2021mean,bjork2014mean}, asset allocation \cite{basak2010dynamic}, decision analytics \cite{luenberger1997investment}, and market design (e.g., electricity \cite{gokgoz2012financial}). In these settings, the variance can be readily computed or approximated, which is different from our case. 
Beyond finance, mean–variance criteria have been embedded in reinforcement learning (RL) and/or Markov decision process (MDP) \citep{mannor2011mean,sobel1994mean,sood2026deep}. For example, \cite{mannor2011mean} analyze finite-horizon MDPs and show NP-hardness with pseudo-polynomial schemes and \cite{sobel1994mean} studies the mean–variance trade-off and characterizes Pareto optima. 

In systems and control, mean–variance formulations have received limited attention, with only a few studies addressing them. \cite{barbieri2020mean} treat infinite-horizon discrete-time linear systems with multiplicative noise via a mean-field decomposition and penalize per-stage variance. \cite{alexanderian2017mean} optimize a mean–variance criterion with a risk aversion parameter in a PDE setting.  \cite{COSTA20082487} formulate a multi-period mean-variance portfolio with Markov-switching parameters while \cite{COSTA2012304} focus on finite-horizon mean-variance optimal control with Markov jumps and multiplicative noise. However, none of these  works provides the exact variance of the cumulative return. 
One most related work is \cite{bijl2016mean}. This work derives exact mean and variance expressions for continuous-time LQG costs via second-moment expansions with Lyapunov/Sylvester operators, and optimizes the variance using gradient descent without convergence guarantees. In contrast, our framework provides a compact random-return-based variance characterization in the discrete-time LQR, enabling weighted mean--variance controller synthesis with stability and convergence guarantees.

Recent advancements in policy gradient methods have demonstrated global convergence and polynomial complexity for the classical LQR, despite the inherent non-convexity of the objective. These results rely critically on the LQR cost landscape being coercive and satisfying gradient dominance~\citep{fazel2018global}. These guarantees have been robustly extended to noisy, finite-horizon environments~\citep{hambly2021policy}. The fundamental difference of sample complexity bounds between model-based and model-free methods for solving the LQR was characterized in ~\cite{tu2019gap}. These insights establish the modern theoretical foundation for data-driven control~\citep{hu2023toward}. However, these standard policy gradient techniques cannot be straightforwardly applied to the mean-variance control problem in our setting, as the requisite coercivity and gradient dominance properties fail to hold. To overcome this limitation, we propose a novel adjoint gradient descent algorithm and formally prove its convergence to a stationary point of the penalized objective.

This paper is organized as follows. Section 2 presents some preliminaries and the problems to be studied. Section 3 derives a closed-form variance expression for the random return. Section 4 presents the mean-variance optimal control design, and Section 5 validates it through case studies. Finally, conclusions are given in Section 6.

%% file: Sec/problem.tex
\section{Preliminary and problem statement}

\textbf{Notation}. Denote by $\mathbb{R}$ the set of real numbers, $\mathbb{C}$ the set of complex numbers, $\mathbb{N}$ the set of natural numbers,  $\mathbb{R}^{n \times m} $ the set of $n \times m$ real-valued matrices, $\mathbb{C}^{n \times m} $ the set of $n \times m$ complex-valued matrices, and $\mathbb{S}^n $ the space of  real symmetric matrices $\in \mathbb{R}^{n\times n}$. Define
$\mathbb{S}^n_+ := \left\{ P \in \mathbb{S}^n \mid x^\top P x \geq 0,\ \forall x \in \mathbb{R}^n \right\}$ and
$\mathbb{S}^n_{++} := \left\{ P \in \mathbb{S}^n \mid x^\top P x > 0,\ \forall x \neq 0 \right\}$. 
For a matrix $P\in \mathbb{R}^{n\times n}$, $\|P\|$ denotes its spectral norm, $\|P\|_F$ its Frobenius norm, $\rho(P)$ its spectral radius, $\mathrm{det}(P)$ its determinant, and $\mathrm{tr}(P)$  its trace. If $P\in \mathbb{S}^n$, 
$\lambda_{\min}(P)$ denotes its minimal eigenvalue and $\lambda_{\max}(P)$ denotes its maximal eigenvalue. 
For a random variable $X$, $\mathbb{E}[X]$ denotes its expectation and $\operatorname{Var}[X]$ denotes its variance. Notation $X_1\mathop{=}\limits^{D} X_2$  denotes that two random variables $X_1,X_2$ are equal in distribution. The subscript $t$
 denotes a physical time index in the system dynamics, whereas $k$ is a series index.

Consider a discrete-time linear system
\begin{equation}\label{Eq:linearsystem}
    x_{t+1}=Ax_{t}+Bu_{t}+v_t
\end{equation}
where \( x_t \in \mathbb{R}^n \), \( u_t \in \mathbb{R}^{m} \), and \( v_t \in \mathbb{R}^n \) are the states, control inputs, and exogenous disturbances, respectively.  The exogenous disturbances \( v_t \) are assumed to be independent and identically distributed (i.i.d.), sampled from an arbitrary distribution \( \mathcal{D} \). Suppose that the covariance matrix of the disturbance is  $\Sigma \in \mathbb{S}^n_{++} $.

\subsection{Canonical discounted LQR}
For the system \eqref{Eq:linearsystem}, discounted LQR has acted as a proxy in RL to synthesize an optimal feedback controller $\pi: \mathbb{R}^n \rightarrow \mathbb{R}^m$ by minimizing the objective
\begin{align*}
V^{\pi}(x) = \mathbb{E} \left[ \sum_{t=0}^{\infty} \gamma^t \left( x_t^\top Q x_t + u_t^\top R u_t \right) \right], 
\end{align*}
where $u_t = \pi(x_t)$, $x_0 = x$, \( Q \in \mathbb{R}^{n \times n} \) and \( R  \in \mathbb{R}^{m \times m} \) are symmetric positive definite matrices, and $\gamma \in (0,1)$ is the  discount factor. 
For a linear stabilizing feedback policy $\pi(x_t) = Kx_t$, we define $V^{K}(x) = V^{\pi}(x)$, which satisfies the following Bellman equation:
\begin{align*}
V^K(x) = &x^\top (Q + K^\top R K) x \\
&+ \gamma \mathbb{E}\left[ V^K(x') \mid x' = (A + B K)x + v_0 \right].
\end{align*}
In particular, $V^K(x)$ admits a quadratic form $$V^K(x)=x^\top Px+q, \ q=\frac{\gamma}{1-\gamma} \operatorname{tr}(P \Sigma)$$
where  $P$ is the  solution to the Lyapunov equation 
\begin{align}\label{eq:lyap_P}
    P=Q+K^\top R K+\gamma A_K^\top P A_K, \ A_K=A+BK.
\end{align}

\subsection{Distributional LQR}
Motivated by the fact that distributional RL provides a comprehensive consideration of the return distribution, distributional LQR has been proposed by \cite{wang2023policy,11020758} to characterize the entire return distribution in linear quadratic control.
For the system \eqref{Eq:linearsystem} with state feedback control $u_t = Kx_t$, the random return from initial state $x$ is defined as
\begin{equation*}
    G^K(x) = \sum_{t=0}^{\infty} \gamma^t  x_t^\top (Q + K^\top R K) x_t, \ x_0=x. 
\end{equation*}
It has been shown in \cite{11020758} that, under the assumption that $K$ is stabilizing, the random return $G^K(x)$ can be analytically characterized as follows:
\begin{align}\label{Eq:Gkx}
      &G^K(x) = x^\top P x + 2 \sum_{k=0}^{\infty} \gamma^{k+1} w_k^\top P A_K^{k+1} x \\  \notag
&+ 
\sum_{k=0}^{\infty} \gamma^{k+1} w_k^\top P w_k + 2 \sum_{k=1}^{\infty} \gamma^{k+1} w_k^\top P \sum_{\tau=0}^{k-1} A_K^{k-\tau} w_\tau,
\end{align}
which is the solution to the random-variable Bellman equation \cite{bellemare2017distributional}:
\begin{align*}\label{eq:rv:Bellman}
    G^{K}(x) & \mathop{=}^{D} x^{\rm{T}}(Q+K^\top R K)x + \gamma G^{K}(x'), \
   x' = A_Kx+v_0.
\end{align*}
Here the auxiliary  random variables $w_k$ follow the same probability distribution as the exogenous disturbance $v_t$ in \eqref{Eq:linearsystem}, that is, $w_k\sim\mathcal{D}$. We assume that $w_k$'s are mutually independent for all $k\in\mathbb{N}$.

\subsection{Problem statement}

Although prior work \cite{wang2023policy,11020758} characterizes and approximates return distributions, leveraging these distributions efficiently for controller synthesis remains underexplored. This paper takes a step forward by designing controllers explicitly grounded in the return distribution. Specifically, we explore designing controllers based on the mean and variance of the random return to trade off the control performance and its variability. The first objective is to derive exact expressions for the return variance, denoted by $\operatorname{Var}\left[G^K(x)\right]$, for both general and specific Gaussian disturbances. 
Using these results, we develop an efficient and convergent algorithm to solve the mean-variance optimal control problem:  $\min_{K} J(K)=\mathbb{E}[G^{K}(x)]+\alpha \operatorname{Var}\left[G^K(x)\right]$,  where $\alpha\geq 0$ is the risk aversion coefficient.

%% file: Sec/variance.tex
\section{Variance of random return $G^K(x)$}
In this section, we provide the exact variance characterization $\operatorname{Var}\left[G^K(x)\right]$ of the random return $G^{K}(x)$. 

\subsection{Non-Gaussian exogenous disturbances}
We consider the class of exogenous disturbances that satisfy the following assumption.

\begin{assumption}\label{Assump:symmetric}
The auxiliary random variables $w_k$ in \eqref{Eq:Gkx} are i.i.d. with zero mean and covariance $\Sigma$, and possess a distribution that is symmetric about the origin\footnote{If $w$ admits a probability density function $p$, this is equivalently interpreted as
$p(w)=p(-w)$ for almost every $w\in\mathbb{R}^n$.
}~\citep{blitzstein_hwang_2019}.
\end{assumption}
\begin{remark}
    This assumption can be extended to the case where the disturbance  distribution is symmetric w.r.t a known point $\mu_w$. By translating the state by the induced equilibrium offset, the dynamics can be rewritten in terms of a disturbance symmetric w.r.t the origin.
\end{remark}
 Let us rewrite the random return as $ G^K(x)= x^\top Px+2\Psi_1+\Psi_2+2\Psi_3$, where 
\begin{eqnarray}\label{eq:simplify}
\begin{cases}
    \Psi_1= \sum_{k=0}^{\infty  } \gamma^{k+1} w_k^\top P A_K^{k+1} x,  \\
    \Psi_2  =\sum_{k=0}^{\infty} \gamma^{k+1} w_k^\top P w_k, \\
     \Psi_3 = \sum_{k=1}^{\infty} \gamma^{k+1} w_k^\top P \sum_{\tau=0}^{k-1} A_K^{k-\tau} w_\tau.
\end{cases}
\end{eqnarray}

\begin{proposition} \label{prop:zerocov}
Under Assumption~\ref{Assump:symmetric}, $\Psi_i$, $i=1,2,3$, are mutually orthogonal, i.e.,
$\mathbb{E}[\Psi_i \Psi_j]=0, \forall i\neq j$.
\end{proposition}
\begin{proof}
    The proof is given in Appendix A. \hspace*{\fill} \qed
\end{proof}
To simplify the variance calculation, we introduce the auxiliary state sequence $\xi_k$, defined by
\begin{equation}\label{Eq:auxiliary_state}
    \xi_{k+1}=A_K\xi_k+w_k,\qquad \xi_0=x.
\end{equation} Its one-step  prediction state is defined by
\[
s_k:=\mathbb{E}[\xi_{k+1}\mid \xi_k]=A_K\xi_k=\xi_{k+1}-w_k. 
\]Note that $s_k$ is determined by $x_0=x$ and linear combinations of the  sequence $\{w_{i}\}_{i=0}^{k-1}$, and $s_k$ is independent of $w_k$.
Then $G^K(x)$ can be reconstructed by combining $\Psi_1$ and $\Psi_3$ as shown in the following proposition.
\begin{proposition}\label{prop:alternativeG}
   An equivalent expression of  $G^K(x)$  is
    \begin{equation*}
        G^K(x)=x^\top Px+ \sum_{k=0}^{\infty} \gamma^{k+1} w_k^\top P w_k +2 \sum_{k=0}^{\infty} \gamma^{k+1} w_k^\top P s_k.
    \end{equation*}
\end{proposition}
\begin{proof}
  Write the auxiliary state 
   {$\xi_{k+1}$}  in \eqref{Eq:auxiliary_state} as ${\xi_{k+1}}=A_K^{k+1}x+ \sum_{\tau=0}^{k}A_K^{k-\tau} w_\tau$.
The definition of $s_k$ yields \(
    s_k=A_K\xi_{k}= A_K^{k+1}x+ \sum_{\tau=0}^{k-1}A_K^{k-\tau} w_\tau.\)
It follows that 
        \begin{align*}
      &G^K(x) = x^\top P x +\sum_{k=0}^{\infty} \gamma^{k+1} w_k^\top P w_k \\ 
&+  2 \sum_{k=0}^{\infty} \gamma^{k+1} w_k^\top P \Big( A_K^{k+1} x+\sum_{\tau=0}^{k-1} A_K^{k-\tau} w_\tau)\\
&=  x^\top P x +\sum_{k=0}^{\infty} \gamma^{k+1} w_k^\top P w_k+ 2 \sum_{k=0}^{\infty} \gamma^{k+1} w_k^\top P s_k. \text{\hspace*{\fill} \qed}
\end{align*}
\end{proof}
Given Proposition \ref{prop:zerocov} and Proposition \ref{prop:alternativeG}, we can characterize the exact variance of $G^K(x)$ in the following theorem. 
\begin{theorem} \label{Theorem:mainresult}  
Let $w$ denote a random variable with the same distribution as $w_k$ (i.e., $w\sim \mathcal{D}$). Under Assumption~\ref{Assump:symmetric}, 
  the exact variance $\operatorname{Var}[G^K(x)]$ of $G^K(x)$ is 
\begin{align}
        &\operatorname{Var}[G^K(x)]= \frac{\gamma^2}{1-\gamma^2} \mathbb{E}\Big[ \Big( 
      {w^\top Pw }\Big)^2\Big]-\frac{\gamma^2}{1-\gamma^2}\mathrm{tr}^2 \left[ {\left( P\Sigma \right)} \right] \nonumber\\
        &\hspace{2cm}+4\gamma^2 \operatorname{tr}(P\Sigma P Z), \label{eq:globalvariance}
    \end{align}
    where  $ Z\in \mathbb{S}^n_{++} $ is the unique solution to the following Lyapunov equation:
   \begin{equation}\label{eq:Pxsigma}
    Z-\gamma^2A_KZA_K^\top =A_K(xx^\top +\frac{\gamma^2}{1-\gamma^2 
    }\Sigma)A_K^\top.
\end{equation}
\end{theorem}
\begin{proof}
The proof is given in Appendix B. \hspace*{\fill} \qed
\end{proof}
Given the fourth-order moment $\mathbb{E}[\|w\|^4]$, Theorem~\ref{Theorem:mainresult} yields an immediate variance bound, as shown in the following corollary.
\begin{corollary}\label{theorem:bound}
    Suppose that  Assumption~\ref{Assump:symmetric} holds and  the random variable $w$ in \eqref{eq:globalvariance} satisfies $\mathbb{E}[\|w \|^4] = \phi^4<+\infty$. Then, the variance of random return  $G^K(x)$ satisfies
       \begin{equation}\label{eq:globalvariancebound}  
    \begin{aligned}
        &\operatorname{Var}[G^K(x)] \leq \frac{\gamma^2}{1-\gamma^2} \|P \|^2 \phi^4-\frac{\gamma^2}{1-\gamma^2}\mathrm{tr}^2 \left[ {\left( P\Sigma \right)} \right] \\
        &\hspace{2cm}+ 4\gamma^2\operatorname{tr}(P\Sigma P Z).
    \end{aligned}
    \end{equation}
\end{corollary}
\begin{proof}
Combining
Theorem~\ref{Theorem:mainresult} and the inequality \(
\mathbb{E} \left[ \left(w^\top P w\right)^2 \right] \leq \mathbb{E}[\| P \|^2 \| w\|^4]=\| P \|^2\phi^4 \)
yields the variance bound in \eqref{eq:globalvariancebound}. \hspace*{\fill} \qed
\end{proof}

\begin{table*}[ht]
\centering
\caption{Comparison of empirical variance and theoretical bounds under different noise types}
\begin{tabular}{lccccccc}
\toprule
\textbf{Noise Type} & \(\phi^4\) & Monte Carlo Variance & Tight Bound &  Norm-Based Bound & Tight/MC  & Norm/MC \\
\midrule
Gaussian      & 139.00     & 1168.50  & 1264.16  & 90084.24 & 1.08 & 77.07 \\
Uniform       & 104.20    & 501.97  & 568.95  & 64860.73 & 1.13 & 129.24 \\
Laplace        & 226.00  & 2841.86  & 3002.19  & 153102.96 & 1.06 & 53.89 \\
Student-\(t\)($\nu=8$) & 208.50 & 2532.43  & 2652.59 & 140430.09 & 1.05 & 55.45 \\
\bottomrule
\end{tabular}
\label{tab:bound_comparison}
\vspace{1mm}
\begin{minipage}{\textwidth}
\small
\textit{Note}: For Gaussian, Uniform, Laplace and Student-t noise, the fourth-order moment \(\phi^4\) is available in closed form.
The last two columns show the ratio between each theoretical bound and the empirical variance.  
The norm-based bound from~\cite{11020758} is originally stated for the second moment \(\mathbb{E}[G^K(x)^2]\); to enable a fair comparison, we subtract \(\mathbb{E}[G^K(x)]^2\) from it.
\end{minipage}
\end{table*} 

We give the  following example to validate the tightness of the above variance bound, compared to the norm-based bound proposed in~\cite{11020758}. 
\begin{example}
To illustrate the tightness of the bound, we introduce a numerical example.
     We consider an idealized data center cooling system    (\cite{9691800}, \cite{dean2020sample}) with the dynamics $x_{t+1}=Ax_t+Bu_t+v_t$, where 
   $A = [
1.01 \ 0.01 \ 0;
0.01 \ 1.01 \ 0.01;
0  \ 0.01 \ 1.01]$ and
$B = I$.
We select $Q=2I$ and $R=I$. Let the initial state $x_0=x= [3,3,3]^\top$, the discount factor $\gamma= 0.8$, and the feedback controller 
$K=-0.02\left[\begin{smallmatrix}
56.19 & 0.7692 & 0.0027\\
0.7692 & 56.20 & 0.7692\\
0.0027 & 0.7692 & 56.19
\end{smallmatrix}\right]$
  We consider four representative types of noise: Gaussian, Uniform, Laplace, and Student-t  distributions, all satisfying Assumption \ref{Assump:symmetric}. The covariance matrices $\Sigma
$ of the four types of noise are set to be identical to $\operatorname{diag}(2,3,4)$.

A quantitative comparison (Table \ref{tab:bound_comparison}) is then made between three quantities to prove the tightness of our bound, compared to the norm-based bound in \cite{11020758} and the empirical variance by Monte Carlo simulations.
From Table \ref{tab:bound_comparison}, compared to the norm-based bound proposed in~\cite{11020758}, the bound in Corollary~\ref{theorem:bound} is significantly tighter. In particular, the ratio between the theoretical bound and the Monte Carlo variance remains consistently below 1.2, indicating that the proposed bound provides a highly accurate characterization of the variance.
This highly improved tightness mainly results from a more accurate treatment of all terms during the bounding process. Instead of applying norm-based relaxations uniformly, the analysis computes most terms without approximations and only applies a loose bound to the fourth-moment-related part. 

In the above numerical example, we consider a controller $K$ such that $\rho(A_K)<1$ and $\lVert A_K\rVert<1$, which is required by the norm-based bound in \cite{11020758}. This requirement stems from the fact that the upper bound in \cite{11020758} depends on 
$\frac{\lVert A_K\rVert^2}{1-\lVert A_K\rVert^2}$,
which is meaningful only when $\lVert A_K\rVert<1$. In contrast, the bound of Corollary~\ref{theorem:bound} remains valid with $\lVert A_K\rVert>1$. 
Here we present a simple numerical example in which 
$K=\left[\begin{smallmatrix}
    &-0.51  &2.99  &0 \\      &-0.01   &-0.51  & -0.01 \\ &0     &-0.01    &-0.51
\end{smallmatrix}\right]$.
Then $\rho(A_K)=0.5<1$, $\lVert A_K \rVert=3.08>1$. We consider the case when $w$ admits a zero-mean Uniform distribution. The empirical variance is $2.09\times 10^{5}$, while our bound is $5.11\times 10^{5}$ and remains valid. The bound-to-variance ratio is 2.46, slightly higher than that in Table~\ref{tab:bound_comparison}.
\end{example}

\subsection{Gaussian exogenous disturbances}
For the Gaussian exogenous disturbances, we can further refine the variance in Theorem~\ref{Theorem:mainresult}.
In this case,  the exogenous disturbances $v_t$ and auxiliary random variables $w_k$  are Gaussian with zero mean and covariance matrix $\Sigma$. 
We begin with the following lemma.

\begin{lemma}\label{Lem:fourmoment}
Let  $w$ be a  zero-mean Gaussian random vector with covariance matrix $\Sigma$. The following holds
\begin{equation*}
     \mathbb{E}\left[ {\left( w^\top P w\right)}^2\right]=2\operatorname{tr}\left( P\Sigma P \Sigma\right)+ \left[ \operatorname{tr}\left( P\Sigma
     \right)\right]^2.
\end{equation*}
\end{lemma}
\begin{proof}
We denote the $i^{th}$ entry of the vector $w$ by $[w]_i$. Then, it follows that 
    \begin{align*}
        &\mathbb{E}\left[ {\left( w^\top P w\right)}^2\right]=\mathbb{E}\left[ {\left( \sum_{i,j}[w]_i P_{ij}[w]_j\right)}^2\right] \\&=\mathbb{E}\left[\sum_{i,j,k,l} P_{ij} P_{kl} [w]_i [w]_j [w]_k [w]_l\right] \\& =\sum_{i,j,k,l} P_{ij} P_{kl}\mathbb{E}\Big[ [w]_i [w]_j [w]_k [w]_l \Big].
    \end{align*}
By Isserlis's theorem \citep{isserlis1918formula}, it follows that 
\begin{align*} \label{eq:modernformIsserlis}
               & \mathbb{E}\Big[ [w]_i [w]_j [w]_k [w]_l \Big]= \mathbb{E}\Big[ [w]_i [w]_j\Big]\mathbb{E}\Big[ [w]_k [w]_l \Big]\\&+ \mathbb{E}\Big[ [w]_i [w]_k\Big]\mathbb{E}\Big[ [w]_j [w]_l \Big] + \mathbb{E}\Big[ [w]_i [w]_l\Big]\mathbb{E}\Big[ [w]_j [w]_k \Big]\\& =\Sigma_{ij}\cdot \Sigma_{kl}+ \Sigma_{ik}\cdot \Sigma_{jl}+\Sigma_{il}\cdot \Sigma_{jk}.
\end{align*}
This further implies that 
  \begin{align*}
&\sum_{i,j,k,l} P_{ij} P_{kl} \notag \mathbb{E}\Big [w]_i [w]_j [w]_k [w]_l \Big] 
\\ \notag &= \sum_{i,j,k,l} P_{ij} P_{kl} \left(\Sigma_{ij} \Sigma_{kl} + \Sigma_{ik}  \Sigma_{jl} + \Sigma_{il}  \Sigma_{jk} \right) \notag \notag \\ \notag
&= {\left( \sum_{i,j} P_{ij} \Sigma_{ij} \right)}^2 
+ \sum_{i,k} \sum_{j,l} P_{ij} P_{kl} \Sigma_{ik} \Sigma_{jl} \\& \notag
+ \sum_{i,l} \sum_{j,k} P_{ij} P_{kl} \Sigma_{il} \Sigma_{jk} \notag = \left[ \operatorname{tr} \left( P \Sigma \right) \right]^2+2 \operatorname{tr} \left( P \Sigma P \Sigma \right).  \text{\hspace*{\fill} \qed}
\end{align*}
\end{proof}
Following Lemma~\ref{Lem:fourmoment},  the next theorem gives an immediate refinement of $\operatorname{Var}\left[G^K(x)\right]$ in the  Gaussian case.  
\begin{theorem}
\label{theorem:variance}
Assume that the auxiliary random variables $w_k$ are i.i.d.\ zero-mean Gaussian random vectors with covariance matrix $\Sigma$.
 Then the exact variance of the random return  $G^K(x)$ is 
    \begin{equation*} \label{eq:variancegaussian}
\begin{aligned}&\operatorname{Var}\left[G^K(x)\right]=\frac{2 \gamma^2}{1 - \gamma^2} \mathrm{tr}(P \Sigma P \Sigma)+ 4\gamma^2\operatorname{tr}(P\Sigma P Z).
        \end{aligned}
\end{equation*}
\end{theorem}
\begin{proof}
    When the random variables $w_k$ are i.i.d.  zero-mean Gaussian with covariance matrix $\Sigma$, Lemma \ref{Lem:fourmoment} implies that $\mathbb{E}[({w^\top P w})^2]=2\operatorname{tr}\left( P\Sigma P \Sigma\right)+ \left[ \operatorname{tr}\left( P\Sigma
     \right)\right]^2$. Therefore, it follows from Theorem~\ref{Theorem:mainresult} that 
         \begin{align*}
              &\operatorname{Var}[G^K(x)]=  \frac{\gamma^2}{1-\gamma^2}\Big[\operatorname{tr}^2 \left( P \Sigma \right) +2 \operatorname{tr} \left( P \Sigma P \Sigma \right)\Big]\\&-\frac{\gamma^2}{1-\gamma^2}\mathrm{tr}^2 \left[ {\left( P\Sigma \right)} \right] +4\gamma^2\operatorname{tr} (P \Sigma P Z)\\& =\frac{2 \gamma^2}{1 - \gamma^2} \mathrm{tr}(P \Sigma P \Sigma)+ 4\gamma^2\operatorname{tr} (P \Sigma P Z).  \tag*{$\qed$} 
         \end{align*}
\end{proof}
\begin{remark}
We give a brief interpretation of $\operatorname{Var}[G^K(x)]$. 
Recall the two terms in $\operatorname{Var}[G^K(x)]$:
\begin{align*}
    \frac{2\gamma^2}{1-\gamma^2}\operatorname{tr}(P\Sigma P\Sigma)
=
\operatorname{Var}\!\left[
\sum_{k=0}^{\infty}\gamma^{k+1}w_k^\top P w_k
\right], \\
4\gamma^2\operatorname{tr}(P\Sigma PZ)
=
\operatorname{Var}\!\left[
\sum_{k=0}^{\infty}2\gamma^{k+1}w_k^\top P s_k
\right].
\end{align*}
Since the auxiliary process $\xi_{k+1}=A_K\xi_k+w_k$ and the system dynamics $x_{t+1}=A_Kx_t+v_t$ share  the same distributions under $w_k\sim v_t$, $s_k:=\mathbb{E}[\xi_{k+1}\mid \xi_k]$
has physical counterparts. Letting
$\hat x_{t+1|t}:=\mathbb E[x_{t+1}\mid x_t]=A_Kx_t$ gives
\begin{align*}
    \sum_{k=0}^{\infty} \gamma^{k+1} w_k^\top P w_k &\stackrel{D}{=} \sum_{t=0}^{\infty} \gamma^{t+1} v_t^\top P v_t,  \\
    \sum_{k=0}^{\infty}2\gamma^{k+1}w_k^\top P s_k
&\stackrel{D}{=}
\sum_{t=0}^{\infty}2\gamma^{t+1}v_t^\top P\hat x_{t+1|t}.
\end{align*}
Thus, the first term corresponds to the quadratic disturbance component related to  the fourth-order moment of the Gaussian disturbance,
whereas the second term quantifies the variance of the infinite-horizon interaction between the noises and the one-step prediction states.
\end{remark}

%% file: Sec/optimalcontrol.tex
\section{Mean-variance optimal control}
Leveraging the variance expression when the noise admits a Gaussian distribution, in this section we formulate the following  mean-variance optimal control problem
\begin{subequations}\label{eq:optform}
\begin{align}
\min_{K}\quad 
& \mathbb{E}\!\left[G^{K}(x)\right]
+ \alpha\, \operatorname{Var}\!\left[G^{K}(x)\right]
\label{eq:optform:obj}\\
\text{s.t.}\quad 
& \rho(A_{K}) < 1
\label{eq:optform:stab}\\
& P - \gamma A_{K}^{\top} P A_{K}
= Q + K^{\top} R K
\label{eq:P_lyap}\\
& Z - \gamma^{2} A_{K} Z A_{K}^{\top}
= A_{K}\!\left(xx^{\top} + \frac{\gamma^{2}}{1-\gamma^{2}}\,\Sigma\right)\!A_{K}^{\top}
\label{eq:Z_lyap}
\end{align}
\end{subequations}
where $\alpha \geq 0 $ is the risk aversion coefficient. Denote by $K^*(\alpha)$ the idealized optimal solution to the original optimization problem~\eqref{eq:optform}.
\begin{remark}
When $\alpha=0$, the proposed original mean-variance formulation reduces exactly to the standard risk-neutral stochastic LQR introduced in Section~2.1. In this case, the optimization problem becomes
$\min_{K}\, \mathbb{E}[G^{K}(x)]$,
which corresponds to the standard risk-neutral stochastic LQR problem. Denote its optimal feedback gain by
$$K_{DARE} = - (R + \gamma B^\top P_{DARE} B)^{-1} \gamma B^\top P_{DARE} A,$$ where
$P_{DARE}\in \mathbb{R}^{n\times n}$ is the solution to the discounted algebraic Riccati equation: 
\begin{equation*}
    \begin{aligned}
      & P_{DARE} = Q + \gamma A^\top P_{DARE} A -  \\&\gamma A^\top P_{DARE} B (R + \gamma B^\top P_{DARE} B)^{-1} B^\top P_{DARE} A. 
    \end{aligned}
\end{equation*}
Equivalently, the optimal solution of the proposed formulation at $\alpha=0$ satisfies $K^*(0)=K_{\mathrm{DARE}}$.
\end{remark}

Note that the mean-variance optimal control is a non-convex problem with both equality and inequality constraints. It is hard to derive a computationally tractable reformulation due to the coupling of two Lyapunov equations. Instead, we develop an \textbf{adjoint gradient descent} method  for efficiently computing the gradients.
Let $\Dstab :=\; \bigl\{K\in\R^{m\times n} \;\big|\;
                  \rho(A_K)<1\bigr\}$ be
the set of  stabilizing linear controllers. For simplicity of notation, we define the objective function
$$
J(K,P,Z) = x^T P x + b \operatorname{tr}(P\Sigma) + c \operatorname{tr}(P\Sigma P\Sigma) + d \operatorname{tr}(P \Sigma P Z),
$$
where   $b=\frac{\gamma}{1-\gamma}, \  
    c =\frac{2\alpha\gamma^2}{1-\gamma^2}$, and $d =4\alpha\gamma^2$. 
We remark that the objective function $J$
 is an explicit function of 
$P$ and $Z$, which in turn depend on $K$
 through the Lyapunov equations. Therefore, $J$
is implicitly a function of $K$
alone, and we write $J(K)$
interchangeably with $J(K,P,Z)$
throughout.

Define the Lagrangian
\begin{align*}
    &\mathcal{L}(K,P,Z,\Lambda_P,\Lambda_Z) \nonumber\\
    &= J(K,P,Z)+\operatorname{tr}\!\bigl(\Lambda_P^{\top} F_P(K,P)\bigr)
       + \operatorname{tr}\!\bigl(\Lambda_Z^{\top} F_Z(K,Z)\bigr), 
\end{align*}
where $\Lambda_P$ and $\Lambda_Z$ are Lagrange multipliers, and 
\begin{align*}
    F_P(K,P) &= P - Q - K^{\top}RK
               - \gamma A_K^{\top}P A_K, \\
    F_Z(K,Z) &= Z - \gamma^2 A_K(Z+M)A_K^{\top},
\end{align*}
with $M=\frac{1}{\gamma^2}xx^\top +\frac{1}{1-\gamma^2}\Sigma$.
To evaluate the gradient of the objective function with respect to the controller gain $K$, we employ the adjoint method. The first step is to determine the sensitivity of the explicit objective function $J(K,P,Z)$ with respect to the state matrices $P$ and $Z$.  
It directly follows that  
\begin{align}
    \frac{\partial J}{\partial P} &= \frac{\partial }{\partial P} \operatorname{tr}(P xx^T) +  b \frac{\partial }{\partial P} \operatorname{tr}(P\Sigma) \nonumber\\
    & + c \frac{\partial }{\partial P} \operatorname{tr}(P\Sigma P\Sigma) + d \frac{\partial }{\partial P} \operatorname{tr}(P \Sigma P Z) \nonumber \\
    &= x x^{\top} + b\Sigma + 2c\,\Sigma P\Sigma
       + d\bigl(\Sigma P Z + Z P\Sigma\bigr), \label{eq:dJdP}  \\
        \frac{\partial J}{\partial Z}
    &= d\frac{\partial }{\partial Z} \operatorname{tr}(P \Sigma P Z)  = d\,P\Sigma P.  \label{eq:dJdZ}
  \end{align}With the established partial derivatives, the core principle of the adjoint method dictates that we choose the Lagrange multipliers, $\Lambda_P$ and $\Lambda_Z$, to render the Lagrangian stationary with respect to the state variables $P$ and $Z$. By enforcing $\partial \mathcal{L}/ \partial P = 0$ and $\partial \mathcal{L}/ \partial Z = 0$, we eliminate the need to compute the complex implicit sensitivities of $P$ and $Z$ with respect to $K$. This condition leads  to the adjoint equations in the following lemma.

\begin{lemma}\label{lem:adjoint} 
For any $K\in \Dstab$, setting $\partial\mathcal{L}/\partial P = 0$ and
$\partial\mathcal{L}/\partial Z = 0$ yields the 
adjoint equations:
\begin{align}
    \Lambda_P - \gamma\,A_K\,\Lambda_P\,A_K^{\top}
    &= -\frac{\partial J}{\partial P},
    \label{eq:adj_P}\\[4pt]
    \Lambda_Z - \gamma^2\,A_K^{\top}\Lambda_Z\,A_K
    &= -\frac{\partial J}{\partial Z}.
    \label{eq:adj_Z}
\end{align}
Both equations have unique solutions $\Lambda_P,\Lambda_Z$
on $\Dstab$.
\end{lemma}
\begin{proof}
Setting $\partial\mathcal{L}/\partial P = 0$ yields that $\frac{\partial J}{\partial P}+\frac{\partial}{\partial P}\,\tr(\Lambda_P^{\top}F_P)=0$, which gives
\begin{align*}
& \frac{\partial J}{\partial P}+\frac{\partial}{\partial P}\,\tr(\Lambda_P^{\top}(P - Q - K^{\top}RK
               - \gamma A_K^{\top}P A_K))=0,        
\end{align*}
and thus 
    $\Lambda_P - \gamma\,A_K\,\Lambda_P\,A_K^{\top}
    = -\frac{\partial J}{\partial P}$. 
Similarly, one can get \eqref{eq:adj_Z}.
For any  \(K\in\mathcal D_{\rm stab}\), we have
\(
\rho(\gamma A_K\otimes A_K)<1 \),
\(\rho(\gamma^2 A_K\otimes A_K)<1.
\)
Therefore, the corresponding Lyapunov operators are nonsingular,
and both equations \eqref{eq:adj_P}--\eqref{eq:adj_Z} admit unique solutions on
\(\mathcal D_{\rm stab}\). 
\hspace*{\fill} \qed
\end{proof}
\begin{remark}
Throughout this paper, the Lyapunov-type equations are evaluated
for stabilizing feedback gains \(K\in\mathcal D_{\rm stab}\), where
\(\rho(A_K)<1\). Hence, for \(a\in\{\gamma^2,\gamma,1\}\),
the corresponding Lyapunov operators are non-singular since
\(
\rho(a A_K\otimes A_K)=a\rho(A_K)^2<1 .
\)
Therefore, the primal equations for \(P\) and \(Z\), the adjoint
equations for \(\Lambda_P\) and \(\Lambda_Z\), and the
barrier-related equations introduced below are all uniquely
solvable on \(\mathcal D_{\rm stab}\).
\end{remark}
Since the adjoint variables $\Lambda_P$ and $\Lambda_Z$ are specifically chosen to satisfy the stationary conditions, the implicit dependencies of $P$ and $Z$ are perfectly decoupled in the differential. Thus the total gradient of $J$ with respect to $K$ simplifies to the partial derivative of the Lagrangian $\mathcal{L}$ with respect to $K$. By evaluating this partial derivative, we obtain a tractable, closed-form expression for the gradient, which is provided in the following proposition. 

\begin{proposition}\label{prop:am}
The gradient of $J$ with respect to $K$
is:
\begin{equation}\label{eq:grad_K}
    \nabla_K J
=
    -2\bigl(RK + \gamma B^{\top}P A_K\bigr)\Lambda_P
   -2\gamma^2\,B^{\top}\Lambda_Z\,A_K(Z+M).
\end{equation}
\end{proposition}
 \begin{proof}
    See Appendix C. \hspace*{\fill} \qed
 \end{proof}

\subsection{Penalized objective}
Note that the discounted Lyapunov equation of $P$ is well-posed up to  $\rho(A_K) = 1/\sqrt{\gamma} > 1$, and the original cost $J(K)$
remains finite as $\rho(A_K) \to 1$. The  sublevel sets of $J$ can accumulate on the boundary of $\mathcal{D}_{\mathrm{stab}}$ and fail to be compact. Now we  introduce 
a \emph{penalized objective}:
\begin{equation}\label{eq:Jmu}
    J_\mu(K) = J(K) + \mu\log\det\widetilde{P}(K),
    \qquad K\in\Dstab.
\end{equation}
where  $\mu>0$ is a barrier weight  and $\widetilde{P}(K)\succ 0$ is the unique
positive definite solution to the \emph{undiscounted} 
Lyapunov equation:
\begin{equation}\label{eq:barrier_lyap}
    \widetilde{P} - A_K^{\top}\widetilde{P}\,A_K = I.
\end{equation}
Adding a penalty  $\log\det\widetilde{P}(K)$ to $J(K)$, which diverges to $+\infty$
 exactly at $\rho(A_K) = 1$, forces  the penalized objective function $J_\mu \to +\infty$ at the boundary of $\partial\mathcal{D}_{\mathrm{stab}}$
and makes every sub-level set compact, which plays an important role for convergence guarantees. 

The next proposition gives the explicit gradient of $\log\det\widetilde{P}(K)$ with respect to $K$. 
\begin{proposition}
\label{prop:barrier_grad}
We have
\begin{equation}\label{eq:grad_barrier}
    \nabla_K\,\log\det\widetilde{P}
 =-2\,B^{\top}\widetilde{P}\,A_K\,\widetilde{\Lambda}.
\end{equation}
where $\widetilde{\Lambda}$ is the unique solution to the 
Lyapunov equation:
\begin{equation}\label{eq:barrier_adj}
    \widetilde{\Lambda} - A_K\,\widetilde{\Lambda}\,A_K^{\top}
    = -\widetilde{P}^{-1}.
\end{equation}
\end{proposition}
\begin{proof}
    See Appendix D. \hspace*{\fill} \qed
\end{proof}

All subsequent analysis (algorithm, convergence) is carried out
on $J_\mu$ rather than on $J$ directly.

\subsection{Adjoint gradient descent algorithm}
The proposed adjoint gradient descent scheme is summarized in Algorithm \ref{alg:agd}.
Starting from a stabilizing LQR gain, each iteration first solves the Lyapunov equations to obtain the matrices \(P\), \(Z\), and the barrier matrix \(\widetilde P\) (lines 3--5). 
The partial derivatives of the objective with respect to the  variables $P,Z$ are then evaluated, and the corresponding adjoint Lyapunov equations are solved to obtain the Lagrange multipliers \(\Lambda_P\), \(\Lambda_Z\), and \(\widetilde\Lambda\) (lines 6--10). 
Finally, the controller is updated using  gradient descent with backtracking Armijo line search, which ensures sufficient decrease of the penalized objective $J_{\mu}$ and keeps the iterates inside $\mathcal{D}_{\rm{stab}}$ (lines 11--16).

\begin{algorithm}[t]
\caption{Adjoint Gradient Descent with Logarithmic Barrier}
\label{alg:agd}
\begin{algorithmic}[1]
\Require $A,B,Q,R,\Sigma,x,\gamma,\alpha,\mu>0$,
         $b$,
         $c$,
         $d$,
         $M$,
         $\beta\in(0,1)$, $\tau\in(0,1)$, $\eta_0>0$, $\varepsilon_{\rm grad}>0$
\Ensure Stationary $K_\mu^*$ with $\rho(A+BK_\mu^*)<1$
\State \textbf{Initialize:} $K_0\leftarrow$ LQR gain (ensures $K_0\in\Dstab$), \(k\leftarrow 0\).
\While{$\|\nabla_K J_\mu(K_k)\|>\varepsilon_{\rm grad}$}
    \State Solve \eqref{eq:lyap_P} $\to P$
    \State Solve \eqref{eq:Pxsigma} $\to Z$
    \State Solve \eqref{eq:barrier_lyap} $\to\widetilde{P}$
    \State Compute $\partial J/\partial P$ via \eqref{eq:dJdP}
    \State Compute $\partial J/\partial Z$ via \eqref{eq:dJdZ}
    \State Solve \eqref{eq:adj_P} $\to\Lambda_P$
    \State Solve \eqref{eq:adj_Z} $\to\Lambda_Z$
    \State Solve \eqref{eq:barrier_adj} $\to\widetilde{\Lambda}$
    \State $\nabla_K J_\mu
           \leftarrow \eqref{eq:grad_K} + \mu\cdot\eqref{eq:grad_barrier}$
    \State $\eta\leftarrow\eta_0$
    \While{$J_\mu(K_k-\eta\nabla_K J_\mu)
           > J_\mu(K_k)-\tau \eta\|\nabla_K J_\mu\|^2$}
        \State $\eta\leftarrow\beta\eta$ \ $\triangleright$ Backtracking-Armijo line search 
    \EndWhile
    \State $\eta_k\leftarrow\eta$
    \State $K_{k+1}\leftarrow K_k - \eta_{k}\,\nabla_K J_\mu$, $k \leftarrow k+1$.
\EndWhile
\State \Return $K_\mu^*\leftarrow K_k$
\end{algorithmic}\label{Alg:Adjoinddescent}
\end{algorithm}

\subsection{Convergence and complexity analysis}
Let $\mathcal{S}_0 = \{K\in\R^{m\times n} \mid J_\mu(K)\le J_\mu(K_0)\}$ where $K_0$ is the initial stabilizing controller in Algorithm~\ref{Alg:Adjoinddescent}. We now show that Algorithm 1 guarantees convergence of $K_k$ to a stationary point of $J_{\mu}$
 at a sublinear rate.

\begin{theorem}\label{thm:conv}
Let $\{K_k\}_{k\ge 0}$ be generated by Algorithm~\ref{alg:agd}. There exists  $\eta_{\min}>0$ such that 
\begin{enumerate}
    \item All iterates satisfy $K_k\in\mathcal{S}_0\subset\Dstab$
          for all $k\ge 0$.
    \item The step sizes satisfy $\eta_k\ge\eta_{\min}>0$
          for all $k\ge 0$.
    \item $\displaystyle\lim_{k\to\infty}\|\nabla J_\mu(K_k)\|=0$.
\end{enumerate}
Furthermore, the best-iterate
gradient satisfies:
\begin{equation}\label{eq:rate}
    \min_{0\le k\le N}\|\nabla J_\mu(K_k)\|^2
    \;\le\;
    \frac{J_\mu(K_0) - \inf_{\mathcal{S}_0}J_\mu}
         {\tau\eta_{\min}\,(N+1)}
    \;=\; O\!\left(\tfrac{1}{N}\right).
\end{equation}
\end{theorem}
\begin{proof}
    See Appendix E. \hspace*{\fill} \qed
\end{proof}

\textbf{Computational Complexity}. The total computation cost of Algorithm~\ref{Alg:Adjoinddescent} is $O(n^3/{\varepsilon_{\rm grad}}^2)$.  Each iteration requires solving six discrete-time Lyapunov equations (three primal and three adjoint), each of size $n \times n$  and costing $O(n^3)$ via the Bartels-Stewart algorithm, plus $O(1)$ matrix multiplications of the same size for assembling the gradient. The Armijo backtracking line search evaluates $J_\mu$ at each trial step, requiring three additional Lyapunov solves per trial. The number of trials is bounded by $\lceil\log_\beta(\eta_0/\eta_{\min})\rceil$, which is a constant. The per-iteration cost is therefore $O(n^3)$.  By the convergence rate, achieving $\min_{k \leq N}\|\nabla J_\mu(K_k)\| \leq \varepsilon_{\rm grad}$ requires $N=O(1/{\varepsilon^2_{\rm grad}})$, giving a total cost of $O(n^3/{\varepsilon^2_{\rm grad}})$.

%% file: Sec/casestudy.tex
\section{Numerical experiments}
\subsection{Case study: vehicle steering control}
Let us consider a discretized linear model for vehicle steering~\cite{aastrom2021feedback,kishida2022risk}. The system dynamics, which describes the vehicle's lateral deviation, is governed by the state-space matrices
$A= [1 \ 0.2;  0 \ 1], \  B=[0.06; 0.20]$.
The state $x = [p, \theta] ^\top$ consists of the lateral position $p$ and the heading angle $\theta$, while the control input $u$ is the steering angle.  
For the mean-variance optimal control problem, the parameters are set as follows:
\begin{equation*}
  Q=10I_2, \, R=1,   \, \gamma=0.8, \, \Sigma=[0.05 \ 0.01;0.01 \ 0.05].
\end{equation*}
We select an initial state $x_0=x=[1,2]^\top$ and $\mu=0.01$.  
We compute the stationary points, denoted as $K_\mu^*(\alpha)$, of  the penalized mean-variance optimization problem \eqref{eq:Jmu} using Algorithm~\ref{Alg:Adjoinddescent}  for a range of risk aversion coefficients $\alpha$, which are shown in Table \ref{tab:alpha_k_j}.
We then analyze and verify  the trade-off property between the theoretical value of $\mathbb{E}[G^{K_\mu^*(\alpha)}(x)]$ and $\operatorname{Var}[G^{K_\mu^*(\alpha)}(x)]$ for different values of $\alpha$ (Figure \ref{fig:steering_overview}). 
We further perform Monte Carlo simulations to generate $N$ trajectories ($N=100000$)  to visualize the probability density of random variable $ G^{K_\mu^*(\alpha)}(x)$ under varying $\alpha$ (Figure \ref{fig:steeringdistribution}).
\begin{table}
\centering
\caption{Controllers obtained by Algorithm \ref{Alg:Adjoinddescent} and objective values}
\label{tab:alpha_k_j}
\begin{tabular}{ccccc}
\toprule
\(\alpha\) & \(K_\mu^*(\alpha)\) & \(J\bigl(K_\mu^*(\alpha)\bigr)\) & \(J(K_{\mathrm{DARE}})\) \\
\midrule
0.00  & \([-1.2641,\,-2.1317]\) & 143.1073   & 143.1073   \\
0.05  & \([-1.5193,\,-2.5861]\) & 170.5056   & 174.7772   \\
0.20  & \([ -1.9204,\,-3.4497]\) & 228.6163   & 269.7868   \\
0.50  & \([-2.4851,\,-4.2551]\) & 303.3749   & 459.8060   \\
1.00  & \([-2.9714,\,-4.5874]\) & 393.7318   & 776.5047   \\
5.00  & \([-3.6057,\,-4.7693]\) & 1008.0674  & 3310.0941  \\
50.00 & \([-3.7927,\,-4.7929]\) & 7759.1516  & 31812.9751 \\
\bottomrule
\end{tabular}
\end{table}

Table ~\ref{tab:alpha_k_j} reports the controllers obtained by Algorithm \ref{Alg:Adjoinddescent}, the corresponding original
objective values \(J(K_\mu^*(\alpha))\), and the baseline costs evaluated at
\(K_{\rm DARE}\). For \(\alpha>0\), the inequality
\(J(K_\mu^*(\alpha))<J(K_{\rm DARE})\) as well as the increasing gap between these two values as \(\alpha\)
grows indicates that the risk-neutral controller becomes progressively
less suitable for the mean-variance objective.

\begin{figure}[t]
    \centering
\includegraphics[width=0.75\linewidth]{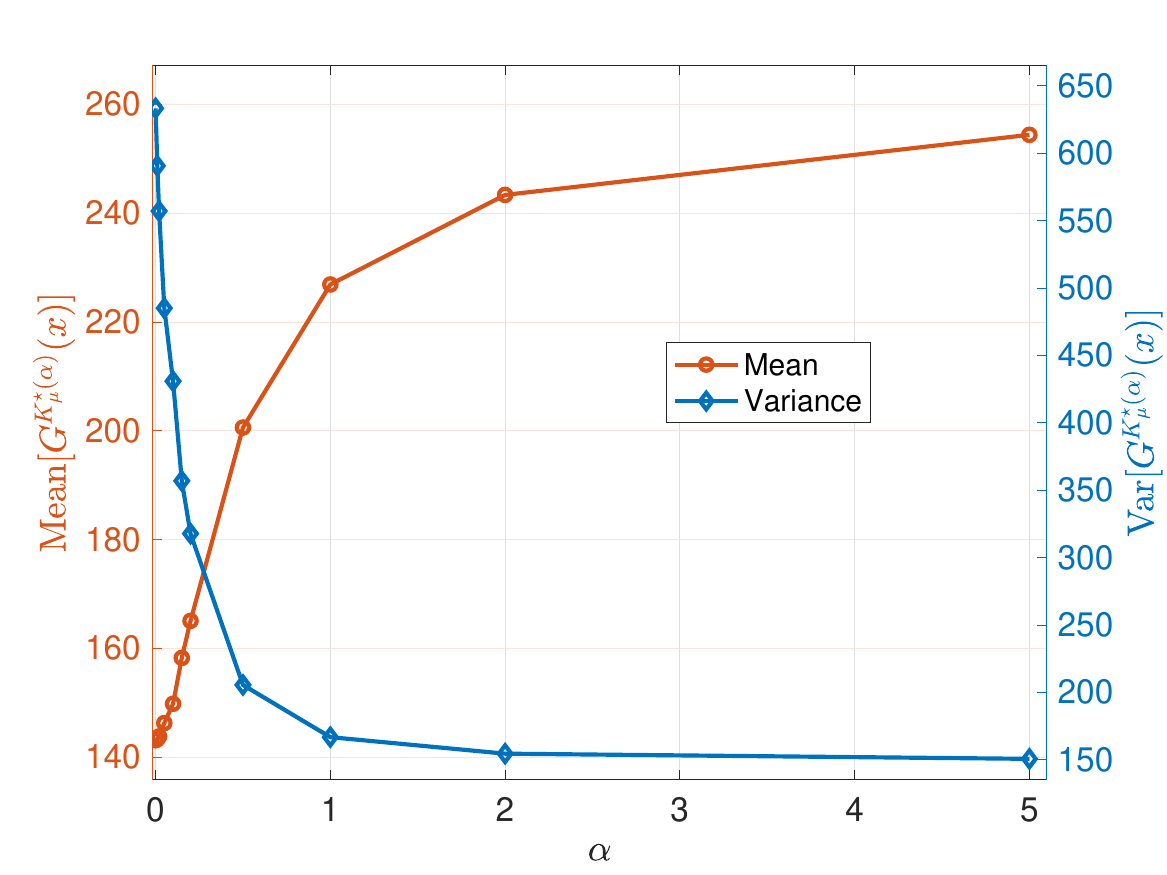}
    \caption{Mean and variance of $G^{K_{\mu}^\star(\alpha)}(x)$ for varying $\alpha$ under the feedback gain $K_\mu^*$.}
    \label{fig:steering_overview}
\end{figure}

Figure~\ref{fig:steering_overview} further explains the tradeoff between the mean and variance of the cumulative return. We can see that $\mathbb{E}[G^{K_\mu^*(\alpha)}]$ increases 
monotonically while $\operatorname{Var}[G^{K_\mu^*(\alpha)}]$  decreases monotonically as shown in Figure \ref{fig:steering_overview}, demonstrating a trade-off achieved by $K_\mu^{*}(\alpha)$ obtained using Algorithm \ref{Alg:Adjoinddescent}.
In addition, when \(0<\alpha<0.2\), the variance decreases sharply from roughly 630 to about 220. When  \(\alpha>2\), the descent decelerates, and the curve almost flattens.  This behavior evidences diminishing marginal gains from additional risk aversion. 
Figure \ref{fig:steeringdistribution} depicts the probability density of $G^{K_{\mu}^*(\alpha)}$ from $\alpha=0$ to $\alpha=50$. As $\alpha$ increases, the distribution progressively shifts rightward while simultaneously becoming more concentrated with lighter tails.
\begin{figure}[t]
    \centering    \includegraphics[width=0.75\linewidth]{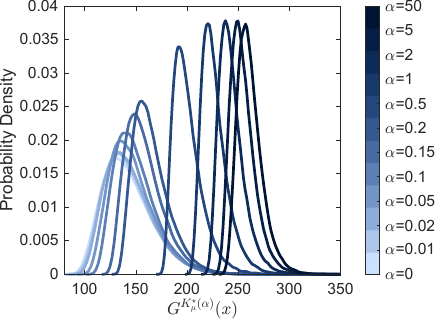}
    \caption{Probability density of $G^{K_{\mu}^*(\alpha)}(x)$ for varying $\alpha$ under the  feedback gain $K_\mu^*$. The color bar denotes different values of the risk aversion coefficient.} 
    \label{fig:steeringdistribution}
\end{figure}

We further provide a physical tradeoff introduced by the weighted variance. Let us define the violation probabilities
$P_p:=P\!\left(\max_{t_{\mathrm{start}}\le t\le T}|p_t|>p_{\max}\right)$, $
P_\theta:=P\!\left(\max_{t_{\mathrm{start}}\le t\le T}|\theta_t|>\theta_{\max}\right)$,
where  $p_t$ is  lateral-position error, $\theta_t$  is  the heading-angle error, and $p_{\max}$ and $\theta_{\max}$ are prescribed thresholds, respectively. These quantities measure the probability that the lateral deviation or heading-angle error exceeds its admissible bound at least once over the time interval $[t_{\mathrm{start}},T]$. We evaluate how the violation probabilities change with respect to~$\alpha$.  Figure \ref{fig:physical_tradeoff} presents two violation probability under the proposed controllers $K^{*}_{\mu}(\alpha)$.  These results reveal that larger values of $\alpha$ reduce the probability of large lateral deviations, whereas the heading-angle violation probability is minimized only for moderate $\alpha$. 
\begin{figure}[t]
    \centering
    \begin{subfigure}[t]{0.48\columnwidth}
        \centering
        \includegraphics[width=\linewidth]{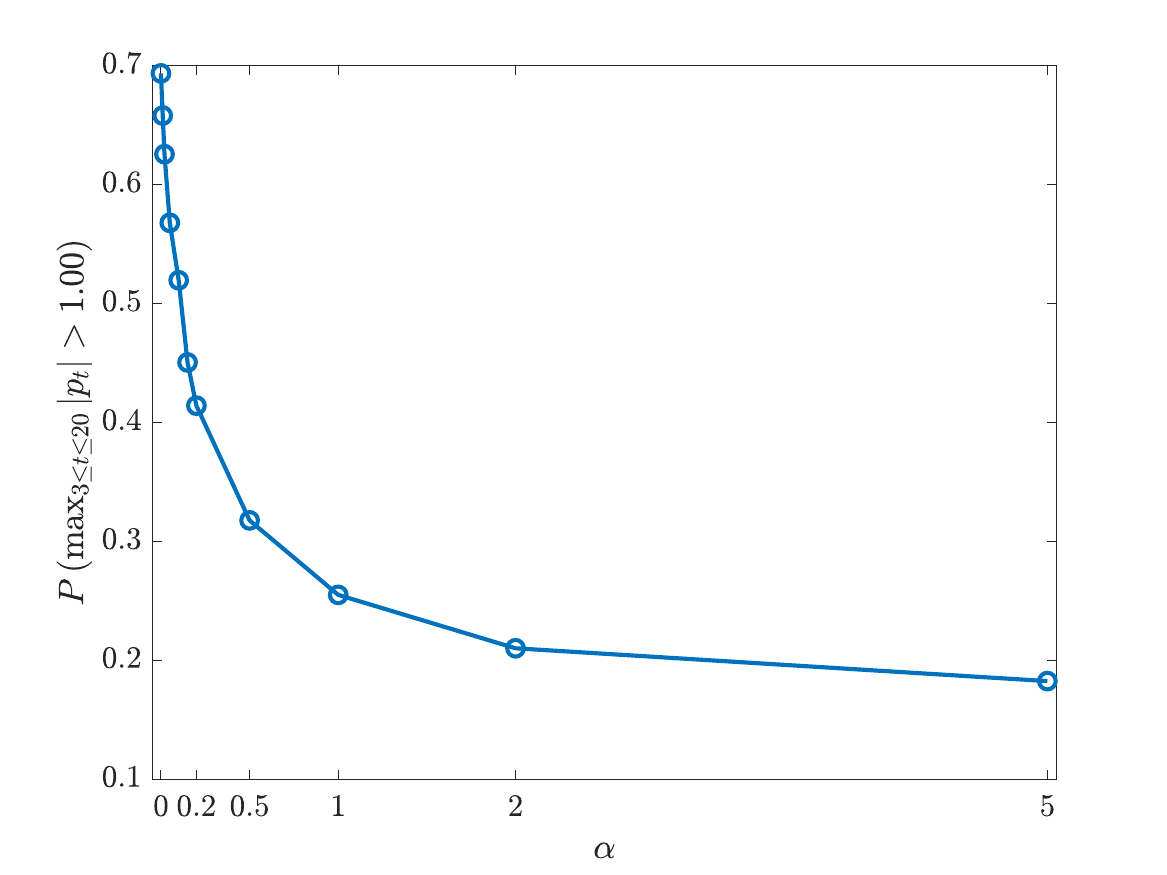}
        \caption{ Lateral-position exceedance probability}
        \label{fig:lateral_exceedance}
    \end{subfigure}
    \hfill
    \begin{subfigure}[t]{0.48\columnwidth}
        \centering
        \includegraphics[width=\linewidth]{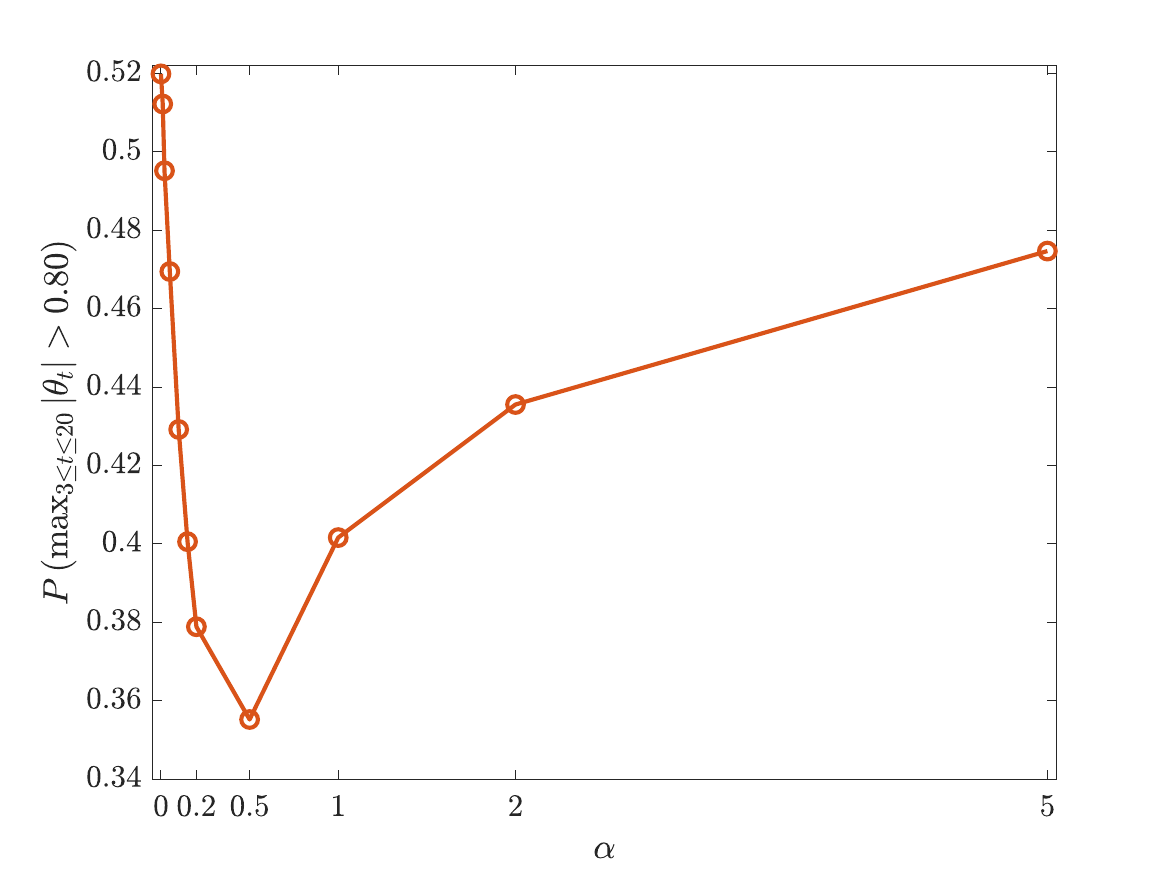}
        \caption{Heading-angle exceedance probability}
        \label{fig:heading_exceedance}
    \end{subfigure}
    \caption{Physical trade-off induced by different values of $\alpha$.}
    \label{fig:physical_tradeoff}
\end{figure}
\subsection{Sensitivity analysis}

Since Algorithm~1 optimizes the penalized objective $J_{\mu}$ rather than the original objective $J$, the  controller obtained for a fixed $\mu>0$ may be biased relative to the solution of $J$. We thus investigate the sensitivity of the stationary solution of $J_{\mu}$ with respect to $\mu$. We use the vehicle steering system in the previous subsection with the same chosen parameters. 

We use a sufficiently small value $\mu_{\mathrm{ref}}=10^{-8}$ as a numerical reference, for which the barrier contribution is negligible. Let $K_{\mu_{\mathrm{ref}}}^{\star}$ be the solution obtained when $\mu=\mu_{\rm ref}$. We consider three sensitivity criteria: (1) the controller deviation is defined by $\Delta_K^{\mu}:=\|K_{\mu}^{\star}-K_{\mu_{\mathrm{ref}}}^{\star}\|_F$; (2) the relative distortion of the original objective $J$ is measured by 
$D_J^{\mu}:=
|J(K_{\mu}^{\star})-J(K_{\mu_{\mathrm{ref}}}^{\star})|
/|J(K_{\mu_{\mathrm{ref}}}^{\star})|$; (3) 
the deviation from first-order stationarity of the original objective is quantified by
$R_J^{\mu}:=\|\nabla J(K_{\mu}^{\star})\|_F$.
\begin{table}[t]
\centering
\caption{Sensitivity analysis with respect to $\mu$
when
$\alpha=0.2$.}
\label{tab:mu_sensitivity}
\resizebox{\columnwidth}{!}{
\begin{tabular}{c c c c}
\hline
$\mu$
&
$D_J^{\mu}$
&
$R_J^{\mu}$
&
$\Delta_K^{\mu}$
\\
\hline
$10^{-4}$
&
$6.129\times 10^{-14}$
&
$2.753\times 10^{-5}$
&
$1.049\times 10^{-6}$
\\
$10^{-3}$
&
$6.142\times 10^{-12}$
&
$2.753\times 10^{-4}$
&
$1.029\times 10^{-5}$
\\
$10^{-2}$
&
$6.192\times 10^{-10}$
&
$2.765\times 10^{-3}$
&
$1.027\times 10^{-4}$
\\
$10^{-1}$
&
$6.183\times 10^{-8}$
&
$2.762\times 10^{-2}$
&
$1.027\times 10^{-3}$
\\
$10^{0}$
&
$6.088\times 10^{-6}$
&
$2.733\times 10^{-1}$
&
$1.020\times 10^{-2}$
\\
$10^{1}$
&
$5.297\times 10^{-4}$
&
$2.486 \times 10^{0}$
&
$9.613\times 10^{-2}$
\\
$10^{2}$
&
$2.353\times 10^{-2}$
&
$1.421\times 10^{1}$
&
$6.910\times 10^{-1}$
\\
\hline
\end{tabular}
}
\end{table}
Table~\ref{tab:mu_sensitivity} presents the numerical analysis over these three criteria. We see that  the stationary solution of $J_{\mu}$ is in general not sensitive to $\mu$ over a broad range of moderate barrier parameters. 
Noticeable sensitivity appears only when $\mu$ becomes sufficiently large. For example, at $\mu=10$ and $\mu=100$, the controller deviation and the original-gradient residual become ‌non-negligible‌.

In practice, one can select $\mu$ using  $D_J^{\mu}$ and $R_J^\mu$. For example, for a fixed $\alpha$, let $\tau_J$ and $\tau_g$ denote the tolerance for $D_J^{\mu}$ and $R_J^\mu$.
Let $\mu_J:=\max\{\mu:D_J^\mu\leq\tau_J\}$ and $\mu_g:=\max\{\mu:R_J^\mu\leq\tau_g\}$. Since both tolerances must be satisfied, the largest admissible barrier parameter is $\mu^\star=\min\{\mu_J,\mu_g\}$, which gives a good admissible barrier effect with prescribed bias.
\subsection{Comparison with LEQG}
Let us further compare the controllers obtained from the proposed mean-variance formulation using Algorithm \ref{Alg:Adjoinddescent} with the linear exponential quadratic Gaussian (LEQG) controllers~\citep{SHAIJU20088773}, which 
solves
\begin{equation}\label{eq:LEQG}
    \min_{K} 
\frac{2}{\theta}\operatorname{log}\mathbb{E}\!\left[  e^{\frac{\theta}{2} G^K(x)} \right],
\end{equation}
where \(G^K(x)\) is the cumulative random return and \(\theta\) is the sensitivity parameter.  
The Taylor expansion gives
\begin{align*}
    &\min_{K} \frac{2}{\theta}
\log \mathbb{E}\!\left[e^{\frac{\theta}{2}G^K(x)}\right]
\\&=
\min_{K}
\left\{
\mathbb{E}[G^K(x)]
+
\frac{\theta}{4}\operatorname{Var}[G^K(x)]
+
O(\theta^2)
\right\}.
\end{align*}
Therefore, in the small-\(\theta\) regime, the mean–variance objective can be interpreted as the first-order approximation of the LEQG objective.
This observation motivates its use as a baseline for comparison.
 \begin{figure}[t]
    \centering
\includegraphics[width=0.75\linewidth]{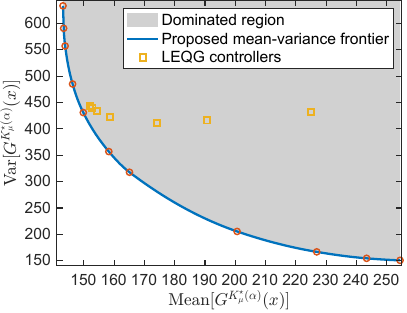}
    \caption{Mean-variance frontier versus LEQG baseline.
   The LEQG controllers are computed using the non-truncated objective in \eqref{eq:LEQG}, rather than its first-order approximation.
    } 
    \label{fig:steeringpareto}
\end{figure}
The comparison is performed over the vehicle steering system  in the previous subsection with the same chosen parameters. 
We compare the resulting controllers in the mean-variance plane induced by the random return.
Figure~\ref{fig:steeringpareto} displays the mean-variance frontier characterized by the proposed controllers across varying values of $\alpha$. The shaded region indicates the set of dominated mean-variance pairs. We observe that the LEQG baseline controllers lie strictly above the frontier, confirming that they are Pareto-inefficient with respect to the mean-variance criterion of the cumulative return $G^K(x)$
 in our setting.

\subsection{Scalability evaluation}
We further evaluate the scalability of our proposed algorithm over a group of systems with different state-input dimensions. For each state-input dimension pair $ (n,m)$, the system matrices $A$ and $B$ are randomly generated.
\begin{figure}[H]
    \centering
\includegraphics[width=0.8\linewidth]{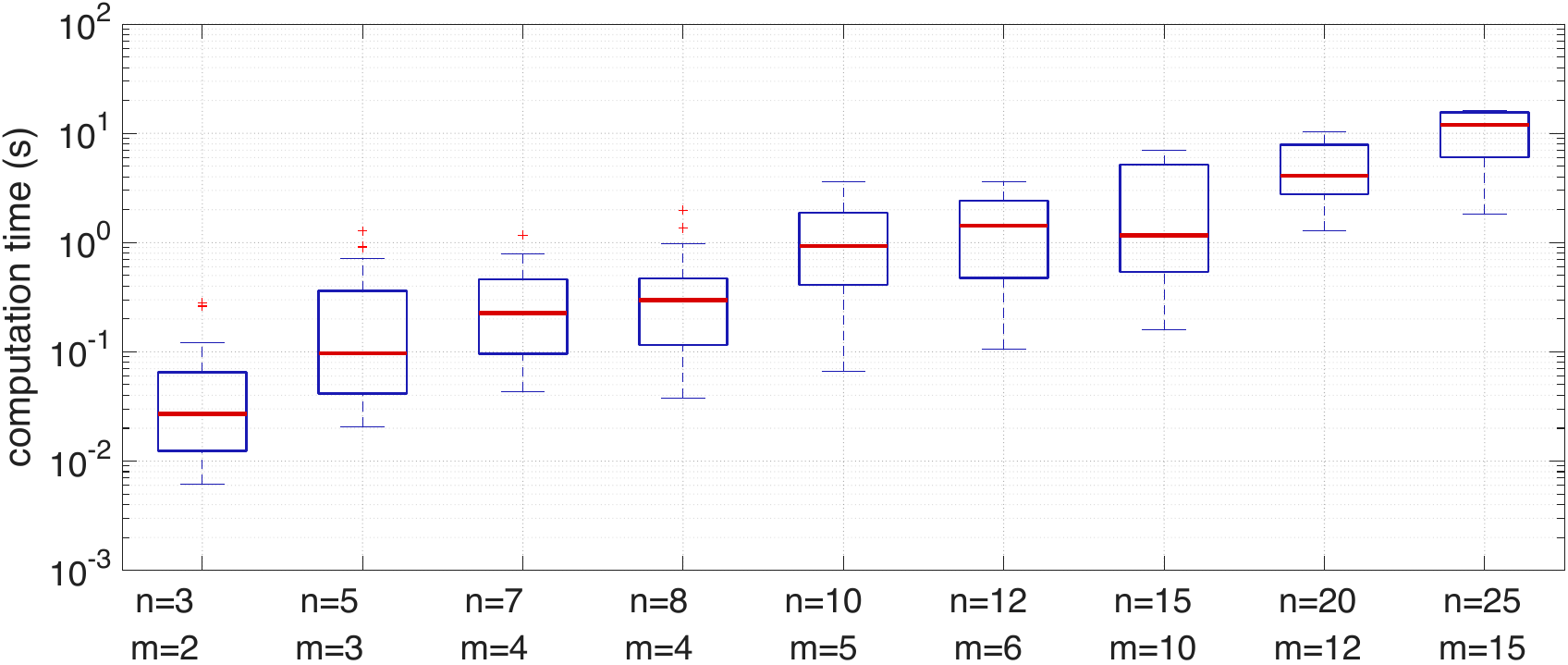}
    \caption{Computation time of adjoint gradient descent algorithm for multiple state and input dimensions.} 
    \label{fig:scalability}
\end{figure}
Fig. \ref{fig:scalability} depicts the computation time in seconds against different system parameters. The vertical axis employs a logarithmic scale ranging from $10^{-2}$ to $10^{2}$ seconds, facilitating the comparison of execution times across distinct levels of complexity. The results are parameterized by $n$, with values spanning from $3$ to $25$ , and $m$, with values ranging from $2$ to $15$. 
We observe a clear upward trend in computational cost as $n$ and $m$ increase, together with a moderate spread across random instances. This behavior is consistent with the increasing complexity of the matrix equations solved within each iteration. Overall, the results suggest that the proposed algorithm exhibits reasonable scalability and remains applicable to systems of substantially higher dimensions.

%% file: Sec/conclusion.tex
\section{Conclusion}
In this paper, we characterized the variance of the cumulative quadratic return in the distributional LQR. We derived a closed-form analytical expression for the variance under the assumption of zero-mean exogenous noise with distributions symmetric about the origin and established a tighter bound for noise with a bounded fourth moment. For zero-mean Gaussian disturbances, the variance simplifies to a more concise form. Building on the closed-form variance, we proposed a mean–variance optimal control framework that penalizes variance using a tunable risk aversion coefficient. In this framework, we developed an adjoint gradient descent method and established its convergence to a stationary point of the penalized objective.
We validated our methodology through numerical experiments.

There are a few interesting future directions to be explored. The current formulation depends on the initial state $x$, so that the objective function can be rewritten in the form of $J(K,x)$. A natural extension is a worst-case formulation  $\min_K \sup_{x \in \mathcal{X}} J(K, x)$
 over a compact uncertainty set $\mathcal{X}$, which introduces the interesting max-min structure to be explored. Another direction is to extend the proposed  mean--variance synthesis framework, together with its stability and convergence analysis, to continuous-time linear systems.

%% file: Sec/Appendix_A.tex
\section*{Appendix A: Proof of Proposition \ref{prop:zerocov}}
Let us begin with a useful lemma before providing the proof of Proposition \ref{prop:zerocov}.
\begin{lemma}\label{lem:oddzero}
Let $w$ denote a random vector with the same distribution as $w_k$ (i.e., $w\sim \mathcal{D}$). Under Assumption~\ref{Assump:symmetric}, 
$\mathbb{E}[ww^\top Pw]=0.$
\end{lemma}
\begin{proof}
Define the mapping $g:\mathbb{R}^{n} \rightarrow \mathbb{R}^n$ where $g(w):=ww^\top Pw=(w^\top Pw)\,w.$
Then \(g\) is an odd mapping, i.e., $g(-w)=-g(w)$.
Since  \(w \overset{D}= -w\) from symmetric definition in \cite{blitzstein_hwang_2019},
we have
\[\mathbb E[g(w)]
=\mathbb E[g(-w)]
=\mathbb E[-g(w)]
=-\mathbb E[g(w)].
\]
Hence, $\mathbb E[ww^\top Pw]=\mathbb E[g(w)]=0$. \hspace*{\fill} \qed
\end{proof}

\textbf{Proof of Proposition \ref{prop:zerocov}}:   The key step in the proof is to exploit the fact that the odd-order moments of the auxiliary random variables $w_k$ are zero under Assumption~\ref{Assump:symmetric}.  Recall the definition of $\Psi_i$ ($i=1,2,3$) in Eq.~\eqref{eq:simplify}. 
    
   Let us begin with $\mathbb{E}[\Psi_1 \Psi_2]$, which can be written as 
    \begin{align*}
        \mathbb{E}[\Psi_1 \Psi_2]=        &\mathbb{E}\Big[\Big( \sum_{i=0}^\infty \gamma^{i+1} w_i^\top P A_K^{i+1} x  \Big)\Big( \sum_{j=0}^\infty \gamma^{j+1} w_j^\top P w_j \Big)\Big] \\
        &= \sum_{i=0}^\infty \gamma^{2i+2}\mathbb{E}\left[(w_i^\top P A_K^{i+1} x)(w_i^\top P w_i)\right]\\&+\sum_{\substack{i,j=0 \\ i \ne j}}^{\infty} 
\gamma^{i+1} \gamma^{j+1} \mathbb{E}\left[
\left( w_i^\top P A_K^{i+1} x \right)
\left( w_j^\top P w_j \right)\right].
    \end{align*}
Given the i.i.d. and zero-mean assumption of $w_k$,  $w_i^\top P A_K^{i+1} x$ is independent of  $w_j^\top P w_j$. This yields 
\begin{align*}
        &\mathbb{E}\left[ \left( w_i^\top P A_K^{i+1} x \right) \left( w_j^\top P w_j \right) \right]\\
        &= \underbrace{\mathbb{E} \left[ w_i^\top P A_K^{i+1} x \right] }_{0}
\mathbb{E} \left[w_j^\top P w_j \right]= 0.
\end{align*}
We then focus on $\mathbb{E}\left[(w_i^\top P A_K^{i+1} x)(w_i^\top P w_i)\right]$. For each $i$, 
\begin{align*}
    &\mathbb{E}\Big[\Big( w_i^\top P A_K^{i+1} x \Big) \Big( w_i^\top P w_i\Big)\Big]= \mathbb{E}\Bigg\{\mathrm{tr}\Big[\Big( x^\top (P A_K^{i+1})^\top \\& w_i w_i^\top P w_i\Big)\Big]\Bigg\}
    =\mathrm{tr}\Bigg\{ x^\top (P A_K^{i+1})^\top\mathbb{E} \Big[ w_i w_i^\top P w_i\Big]\Bigg\}.
\end{align*}
We then refer to Lemma \ref{lem:oddzero} and  exploit the i.i.d. property, $\mathbb{E}\Big[ w_i w_i^\top P w_i\Big]=\mathbb{E}\Big[ w w^\top P w\Big]=0$ for all $i \in \mathbb{N}$.
Note that 
$\mathbb{E}[\Psi_1 \Psi_2]$ is just the sum of all zero terms, so we have \(
    \mathbb{E}[\Psi_1 \Psi_2]=\sum_{i=0}^\infty \gamma^{2i+2}\mathbb{E}\left[(w_i^\top P A_K^{i+1} x)(w_i^\top P w_i)\right]=0.\)
Next, we consider  $\mathbb{E}[\Psi_2 \Psi_3]$, which follows that 
\begin{align*}
    &\mathbb{E}\Big[\Big( \sum_{i=0}^\infty \gamma^{i+1} w_i^\top P w_i \Big)\Big( \sum_{j=1}^\infty \gamma^{j+1} w_j^\top P \sum_{\tau=0}^{j-1} A_K^{j-\tau} w_\tau  \Big) \Big] \\& =\mathbb{E} \left[
\sum_{i=0}^{\infty} \sum_{j=1}^{\infty} 
\gamma^{i+1} \gamma^{j+1}
\left( w_i^\top P w_i \right)
\left( w_j^\top P \sum_{\tau=0}^{j-1} A_K^{j - \tau} w_\tau \right)
\right].
\end{align*}
Define $\psi(i,j):=\mathbb{E}\left[\left( w_i^\top P w_i \right)
\left( w_j^\top P \sum_{\tau=0}^{j-1} A_K^{j - \tau} w_\tau \right)\right]$.  We consider two distinct cases for the value of $\psi(i,j)$. 

\begin{enumerate}
    \item [(i)]$i=j$: It follows that 
    \begin{align*}
    &\psi(i,i)=\mathbb{E}\Bigg[ \left( w_i^\top P w_i \right) \left(   w_i^\top P \sum_{\tau=0}^{i-1} A_K^{i-\tau} w_\tau \right)\Bigg] 
    \\&=\mathbb{E}\Big[ w_i^\top P w_i w_i^\top P \Big] \underbrace{\mathbb{E}\Bigg[\sum_{\tau=0}^{i-1} A_K^{i-\tau} w_\tau\Bigg]}_{0}=0.
    \end{align*}
    \item [(ii)] $i\neq j$: Since $\tau \neq j$, $i\neq j$, 
    $w_j^\top$ is independent of $w_i^\top P w_i$ and $\sum_{\tau=0}^{j-1} A_K^{j-\tau} w_\tau$. Therefore, we have
    \begin{align*}
            &\psi(i,j)=\mathbb{E}\Bigg[ \left(   w_j^\top P \sum_{\tau=0}^{j-1} A_K^{j-\tau} w_\tau \right)\left( w_i^\top P w_i \right) \Bigg]\\
            &=\underbrace{\mathbb{E}[w_j^\top]}_{0}P \mathbb{E}\left[\left(  \sum_{\tau=0}^{j-1} A_K^{j-\tau} w_\tau \right)\left( w_i^\top P w_i \right)\right]=0.
    \end{align*}
\end{enumerate}
Therefore,  $\psi(i,j)$ = 0  holds  for all $ i, j$ , as established by case-wise analysis.
Therefore, we have \(\mathbb{E}[\Psi_2\Psi_3]=\sum_{i=0}^{\infty} \sum_{j=1}^{\infty} 
\gamma^{i+1} \gamma^{j+1}\psi(i,j)=0.\)
Finally, we consider  $\mathbb{E}[\Psi_1 \Psi_3]$ and rewrite its summation:
\begin{align*}
    &\mathbb{E}\Big[\Big( \sum_{i=0}^\infty \gamma^{i+1} w_i^\top P A_K^{i+1} x \Big)\Big( \sum_{j=1}^\infty \gamma^{j+1} w_j^\top P \sum_{\tau=0}^{j-1} A_K^{j-\tau} w_\tau  \Big) \Big] 
    \\&=\mathbb{E} \left[
\sum_{i=0}^{\infty} \sum_{j=1}^{\infty} 
\gamma^{i+j+2} 
\left( w_i^\top P A_K^{i+1} x \right)
\left( w_j^\top P \sum_{\tau=0}^{j-1} A_K^{j - \tau} w_\tau \right)
\right].
\end{align*}
Define $\Gamma(i,j):=\mathbb{E}\left[  \left( w_i^\top P A_K^{i+1} x \right)
\left( w_j^\top P \sum_{\tau=0}^{j-1} A_K^{j - \tau} w_\tau \right)\right]$.
\begin{enumerate}
    \item [(i)] $i=j$: It follows that
    \begin{align*}
    &\Gamma(i,i)=\mathbb{E}\Bigg[ \left( w_i^\top P A_K^{i+1} x \right) \left(   w_i^\top P \sum_{\tau=0}^{i-1} A_K^{i-\tau} w_\tau \right)\Bigg] 
    \\&={\mathrm{tr}\Big [ x^\top  (A_K^{i+1})^\top P \Sigma \underbrace{\mathbb{E}\left(\sum_{\tau=0}^{i-1} A_K^{i-\tau} w_\tau\right)}_{0} \Big]} =0.
    \end{align*}
    \item [(ii)] $i\neq j$:
    Since $\tau \neq j$, $i\neq j$, 
    $w_j^\top$ is independent of $w_i^\top P A_K^{i+1}x$ and $\sum_{\tau=0}^{j-1} A_K^{j-\tau} w_\tau$. Therefore, we have
    \begin{align*}
            &\Gamma(i,j)=\mathbb{E}\Bigg[ \left(   w_j^\top P \sum_{\tau=0}^{j-1} A_K^{j-\tau} w_\tau \right)\left(w_i^\top P A_K^{i+1} x  \right) \Bigg]\\
            &=\underbrace{\mathbb{E}[w_j^\top]}_{0}P \mathbb{E}\left[\left(  \sum_{\tau=0}^{j-1} A_K^{j-\tau} w_\tau \right)\left( w_i^\top P A_K^{i+1} x  \right)\right]=0.
    \end{align*}
\end{enumerate}
Therefore,  $\Gamma(i,j)=0$ 
holds  for all $ i, j$, which implies \(    \mathbb{E}[\Psi_1\Psi_3]=\sum_{i=0}^{\infty} \sum_{j=1}^{\infty} 
\gamma^{i+1} \gamma^{j+1}\Gamma(i,j)=0.\)
In conclusion, $\mathbb{E}[\Psi_1 \Psi_2]=\mathbb{E}[\Psi_1 \Psi_3]=\mathbb{E}[\Psi_2 \Psi_3]=0$, thereby completing the proof of Proposition~\ref{prop:zerocov}.
\hspace*{\fill} \qed

%% file: Sec/Appendix_B.tex
\section*{Appendix B: Proof of Theorem \ref{Theorem:mainresult}}
To prove Theorem 1, we need the following lemma. 
\begin{lemma}\label{lem:expectation}
    The expected value of the random variable $G^{K}(x)$ is $x^\top P x +\frac{\gamma}{1-\gamma}\operatorname{tr}(P\Sigma)$.
\end{lemma}

\begin{proof}
Since the value function in canonical LQR is denoted as the expectation of random return,
    $\mathbb{E}[G^K(x)]=V^{K}(x)=x^\top Px+\frac{\gamma}{1-\gamma}\operatorname{tr}(P\Sigma
    )$ . \hspace*{\fill} \qed
\end{proof}

\textbf{Proof of Theorem \ref{Theorem:mainresult}}:
Recalling Proposition \ref{prop:alternativeG}, we have
\begin{equation*}
    G^K(x)=x^\top Px+ \underbrace{\sum_{k=0}^{\infty} \gamma^{k+1} w_k^\top P w_k}_{:=\Gamma_1} +\underbrace{2 \sum_{k=0}^{\infty} \gamma^{k+1} w_k^\top P s_k}_{:=\Gamma_2}.
\end{equation*}
Note that $\Gamma_1=\Psi_2$ and $\Gamma_2=2(\Psi_1+\Psi_3)$. Based on Proposition \ref{prop:zerocov}, $\Gamma_1$ and $\Gamma_2$ are orthogonal. Therefore, 
the variance of $G^K(x)$ is:
\begin{equation*}
    \operatorname{Var}[G^K(x)]=\operatorname{Var}[\Gamma
    _1]+\operatorname{Var}[\Gamma_2].
\end{equation*}
We first focus on  $\operatorname{Var}[\Gamma_1] $.
Using the result $\mathbb{E}[\Gamma
_1]=\frac{\gamma}{1-\gamma}\operatorname{tr}(P\Sigma)$ from Lemma \ref{lem:expectation}, we have
    \begin{align}
        &\operatorname{Var}[\Gamma_1]=\mathbb{E}[\Gamma_1^2]-\mathbb{E}[\Gamma_1]^2=\mathbb{E}\Bigg[\left( \sum_{i=0}^\infty \gamma^{i+1} w_i^\top P w_i \right)\nonumber\\
        &\left( \sum_{j=0}^\infty \gamma^{j+1} w_j^\top P w_j \right)\Bigg]-\left[\frac{\gamma}{1-\gamma}\operatorname{tr}(P\Sigma) \right]^2. \label{eq:Gamma1proto}
    \end{align}
We then focus on calculating $\mathbb{E}[\Gamma_1^2]$.
In \eqref{eq:Gamma1proto}, when $i = j$, the square of the quadratic form of the auxiliary random variables, $\mathbb{E}[(w_i^\top P w_i)^2]$, appears. However, when $i \neq j$, the products of the quadratic forms of auxiliary random variables appear at different series indices, and no fourth-order term arises. Therefore, we rewrite the summation in \eqref{eq:Gamma1proto} by computing separately for the cases where $i=j$ and where $i\neq j$.
\begin{enumerate}
    \item [(i)] $i=j$: The diagonal summation of the matrix series is as follows:
    \begin{align}\label{eq:Gamma1diagonal}
            \mathbb{E}\left[ \sum_{i=0}^{\infty} \gamma^{2i+2} \left( w_i^\top P w_i\right)^2\right]=\frac{\gamma^2}{1-\gamma^2}\mathbb{E}\Big[ \Big(  w_i^\top Pw_i \Big)^2\Big].
    \end{align}
    \item [(ii)] $i\neq j$: Since $w_i^\top P w_i$ is independent of $w_j^\top P w_j$, 
    the non-diagonal summation of the matrix series is as follows:
   \begin{align}
&\mathbb{E}\Bigg[
\Big( \sum_{i=0}^\infty \gamma^{i+1} w_i^\top P w_i \Big)
\Big( \sum_{\substack{j=0\\ j\ne i}}^\infty
\gamma^{j+1} w_j^\top P w_j \Big)
\Bigg] \nonumber\\
&=
\left[
\left(\frac{\gamma}{1-\gamma}\right)^2
-
\frac{\gamma^2}{1-\gamma^2}
\right]
\mathbb{E}\!\left[w_i^\top P w_i\right]
\mathbb{E}\!\left[w_j^\top P w_j\right] \nonumber\\
&=
\left[
\left(\frac{\gamma}{1-\gamma}\right)^2
-
\frac{\gamma^2}{1-\gamma^2}
\right]
\operatorname{tr}^2(P\Sigma).
\label{eq:Gamma1nondiagonal}
\end{align}
\end{enumerate}

Adding the results of Eqs. \eqref{eq:Gamma1diagonal} and \eqref{eq:Gamma1nondiagonal} gives $\mathbb{E}[\Gamma_1^2]$. Therefore, we obtain $\operatorname{Var}[\Gamma_1]$.
\begin{align*}
        &\operatorname{Var}[\Gamma_1]= \left[\left( \frac{\gamma}{1 - \gamma} \right)^2 - \frac{\gamma^2}{1 - \gamma^2}\right] \mathrm{tr}^2  \left( P\Sigma \right)\\&+\frac{\gamma^2}{1-\gamma^2}\mathbb{E}[(w_i^\top P w_i)^2]-\left[\frac{\gamma}{1-\gamma}\operatorname{tr}(P\Sigma) \right]^2 \\
        &=\frac{\gamma^2}{1-\gamma^2}\mathbb{E}[(w_k^\top P w_k)^2]- \frac{\gamma^2}{1 - \gamma^2} \mathrm{tr}^2  \left( P\Sigma \right).
\end{align*}
We then focus on the variance of $\Gamma_2$. Since $s_k$ is independent of $w_k$, we have $\mathbb{E}[\Gamma_2]=0$. Then, it follows that
\begin{align}
        \operatorname{Var}[\Gamma_2]&=\mathbb{E}[\Gamma_2^2]=4\gamma^2\mathbb{E}\Big[\Big(\sum_{k=0}^{\infty}\gamma^k w_k^\top P s_k \Big)^2\Big]  \nonumber \\
        &=4\gamma^2\mathbb{E}\Big[ \Big(\sum_{i=0}^{\infty}\gamma^i w_i^\top P s_i \Big)\Big(\sum_{j=0}^{\infty}\gamma^j w_j^\top P s_j \Big) \Big]. \label{eq:Gamma2proto}
\end{align}
We focus on the case where $i\neq j$ in Eq.~\eqref{eq:Gamma2proto}: 
\begin{equation*}
   \mathbb{E}\Big[(\gamma^i w_i^\top P s_i)(\gamma^j w_j^\top P s_j)\Big]=\gamma^{i+j}\mathbb{E}[( w_i^\top P s_i)( w_j^\top P s_j)].
\end{equation*}
When $i>j$, $w_i$ is independent of $w_j$ because of i.i.d. property. 
From the definition of $s_k$, $w_i$ is independent of $s_i$. Besides, $w_i$ is independent of $s_j$ because $s_j$ is a history prediction of  state $\xi_j$ at series index $j$ before series index~$i$. As a result, $w_i$ is independent of $w_j$, $s_i$, $s_j$, which yields
\begin{equation*}
    \mathbb{E}[( w_i^\top P s_i)( w_j^\top P s_j)]=\underbrace{\mathbb{E}[w_i^\top]}_{0}\mathbb{E}[P s_iw_j^\top Ps_j]=0.
\end{equation*} 
By symmetry between $i$ and $j$, the situation for $i<j$ is identical to $i>j$. Therefore, we have  $\mathbb{E}[( w_i^\top P s_i)( w_j^\top P s_j)]=0$, for all $i\neq j$. Then, we have
\begin{align*}
&\operatorname{Var}[\Gamma_2]
=4\gamma^2\sum_{k=0}^{\infty}\gamma^{2k}
\mathbb{E}\!\left[(w_k^\top P s_k)^2\right] 
=4\gamma^2\sum_{k=0}^{\infty}\gamma^{2k}\\
&\operatorname{tr}\!\left(P\Sigma P\,\mathbb{E}[s_k s_k^\top]\right)
=4\gamma^2\operatorname{tr}\!\left(
P\Sigma P\sum_{k=0}^{\infty}\gamma^{2k}
\mathbb{E}[s_k s_k^\top]
\right).
\end{align*}
To this end, we have 
\begin{align}
    &\operatorname{Var}[G^K(x)]=\operatorname{Var}[\Gamma_1]+\operatorname{Var}[\Gamma_2]= \frac{\gamma^2}{1-\gamma^2}\mathbb{E}[(w_k^\top P w_k)^2] \notag
    \\&- \frac{\gamma^2}{1 - \gamma^2} \mathrm{tr}^2  \left( P\Sigma \right)+4\gamma^2 \operatorname{tr}\Big\{ \sum_{k=0}^{\infty}\gamma^{2k}P\Sigma
    P\mathbb{E}[ s_ks_k^\top ]  \Big\}. \label{eq:44}
\end{align}

Define $Z':=\sum_{k=0}^{\infty}\gamma^{2k}\mathbb{E}[s_ks_k^\top].$ 
To give the variance of $G^K(x)$ in \eqref{eq:globalvariance} defined by $Z$, we need to prove that $Z$, the solution to a Lyapunov equation \eqref{eq:Pxsigma}, is exactly $Z'$. Recall the definition of  $\xi_k$ in \eqref{Eq:auxiliary_state}. 
To show this, let us consider the recursive relationship between $s_{k+1}$ and $s_k$:
    \begin{equation*}
s_{k+1}=A_K{\xi_{k+1}}=A_K(A_K\xi_k+w_k)=A_Ks_k+A_Kw_k.
    \end{equation*}
    Then, we represent the recursive relationship between $\mathbb{E}[s_ks_k^\top ]$ and $\mathbb{E}[s_{k+1}s_{k+1}^\top ]$:
    \begin{align}
        &\mathbb{E}[s_{k+1}s_{k+1}^\top ]= \mathbb{E}[(A_Ks_k+A_Kw_k)(A_Ks_k+A_Kw_k)^\top] \nonumber\\
        &=A_K\mathbb{E}[s_ks_k^\top]A_K^\top +A_K\Sigma A_K^\top, \label{eq:recursives}
        \end{align}
    where the cross terms  vanish since $s_k$ and $w_k$ are independent.
    The initial  $\mathbb{E}[s_0s_0^\top ]$ is $A_K xx^\top A_K^\top$. Recalling the subtraction in \eqref{eq:Pxsigma}, we have the following.
    \begin{align*}
       &\sum_{k=0}^{\infty}\gamma^{2k}\mathbb{E}[s_ks_k^\top]-\gamma^2A_K\sum_{k=0}^{\infty}\gamma^{2k}\mathbb{E}[s_ks_k^\top]A_K^\top \\
       &=A_K xx^\top A_K^\top+ \sum_{k=1}^{\infty}\gamma^{2k}\mathbb{E}[s_ks_k^\top]
    \\&-A_K\sum_{k=1}^{\infty}\gamma^{2k}\mathbb{E}[s_{k-1}s_{k-1}^\top]A_K^\top.
    \end{align*}
    Using the recursive relationship in \eqref{eq:recursives}, we have
\begin{align}
&\underbrace{\sum_{k=0}^{\infty}\gamma^{2k}\mathbb{E}[s_ks_k^\top]}_{Z'}-\gamma^2A_K\underbrace{\sum_{k=0}^{\infty}\gamma^{2k}\mathbb{E}[s_ks_k^\top]}_{Z'}A_K^\top \nonumber
     \\&= A_Kxx^\top A_K^\top+
\sum_{k=1}^{\infty}\gamma^{2k}\underbrace{\Big( \mathbb{E}[s_ks_k^\top]- A_K\mathbb{E}[s_{k-1}s_{k-1}^\top ] A_K^\top \Big)}_{A_K\Sigma A_K^\top} \nonumber\\
        &= A_K (xx^\top+\frac{\gamma^2}{1-\gamma^2}\Sigma) A_K^\top. \label{eq:49}
    \end{align}
Comparing \eqref{eq:49} with \eqref{eq:Pxsigma}, we conclude that $Z = Z'$.
Substituting $Z$ into \eqref{eq:44} completes the proof. \hspace*{\fill} \qed

%% file: Sec/Appendix_C.tex
\section*{Appendix C: Proof of Proposition~\ref{prop:am}}
Once $\Lambda_P$ and
$\Lambda_Z$ are chosen to satisfy the adjoint equations
\eqref{eq:adj_P}--\eqref{eq:adj_Z} (i.e.\ to eliminate the
implicit $K$-dependence through $P$ and $Z$), the total derivative
of $J$ along any perturbation $\mathrm{d}K$ equals the
\emph{explicit} partial derivative of the Lagrangian:
    \begin{align*}
   & \mathrm{d}J
    = \tr(\mathrm{d}K^\top \frac{\partial \mathcal{L}}{\partial K}) = \tr(\mathrm{d}K^\top  \frac{\partial J}{\partial K}\bigg|_{P,Z} ) \nonumber \\
     & + \tr\!\Bigl(\mathrm{d}K^{\top}\,
\Bigl[\frac{\partial\,\tr(\Lambda_P^{\top}F_P)}{\partial K}
        \Bigr]\Bigr)
      + \tr\!\Bigl(\mathrm{d}K^{\top}\,\Bigl[\frac{\partial\,\tr(\Lambda_Z^{\top}F_Z)}{\partial K}
        \Bigr]\Bigr).\label{eq:total_deriv} 
\end{align*}
Since $J$ has no explicit dependence on $K$ (it depends on $K$ only
implicitly through $P$ and $Z$), $\frac{\partial J}{\partial K}\bigg|_{P,Z}=0$.
The gradient is then identified by $\mathrm{d}J
= \tr(\mathrm{d}K^{\top}\nabla_K J)$, so
$\nabla_K J = \frac{\partial\,\tr(\Lambda_P^{\top}F_P)}{\partial K}
+ \frac{\partial\,\tr(\Lambda_Z^{\top}F_Z)}{\partial K}$.

Let us begin with $\frac{\partial\,\tr(\Lambda_P^{\top}F_P)}{\partial K}$. Recall that
$F_P = P - Q - K^{\top}RK - \gamma A_K^{\top}PA_K$.
With $P$ fixed and $\mathrm{d}A_K = B\,\mathrm{d}K$:
\begin{align*}
    \mathrm{d}F_P\big|_P
    &= -\mathrm{d}K^{\top}RK - K^{\top}R\,\mathrm{d}K
       \\
       &\quad - \gamma\bigl(\mathrm{d}K^{\top}B^{\top}PA_K
         + A_K^{\top}PB\,\mathrm{d}K\bigr).
\end{align*}
It then follows that 
\begin{align*}
    \tr(\Lambda_P^{\top}\,\mathrm{d}F_P)
    &= -\tr\!\bigl(\mathrm{d}K^{\top}RK\Lambda_P\bigr)
       -\tr\!\bigl(\mathrm{d}K^{\top}(K^{\top}R)^{\top}\Lambda_P\bigr) \\
    &\quad
       -\gamma\,\tr\!\bigl(\mathrm{d}K^{\top}B^{\top}PA_K\Lambda_P\bigr)
      \\
     & \quad -\gamma\,\tr\!\bigl(\mathrm{d}K^{\top}
         (A_K^{\top}PB)^{\top}\Lambda_P\bigr).
\end{align*}
Putting them together gives that  
\begin{equation}\label{eq:grad_FP}
    \frac{\partial\,\tr(\Lambda_P^{\top}F_P)}{\partial K}
    = -2(RK + \gamma B^{\top}PA_K)\Lambda_P.
\end{equation}
One can apply similar steps  to $\frac{\partial\,\tr(\Lambda_Z^{\top}F_Z)}{\partial K}$, which gives 
\begin{equation}\label{eq:grad_FZ}
    \frac{\partial\,\tr(\Lambda_Z^{\top}F_Z)}{\partial K}
    = -2\gamma^2\,B^{\top}\Lambda_Z A_K(Z+M).
\end{equation}
From \eqref{eq:grad_FP}  
and \eqref{eq:grad_FZ}, we have  
\[
    \nabla_K J
    = -2(RK + \gamma B^{\top}PA_K)\Lambda_P
      - 2\gamma^2\,B^{\top}\Lambda_Z A_K(Z+M).
\]
We complete the proof.     \textbf{\hspace*{\fill} \qed}

%% file: Sec/Appendix_D.tex
\section*{Appendix D: Proof of Proposition~\ref{prop:barrier_grad}}
Let $\phi:=\log\det\widetilde{P}(K)$. 
We will apply the adjoint method to the  function
$\log\det\widetilde{P}(K)$
subject to the constraint $H(K,\widetilde{P}) =
\widetilde{P} - A_K^{\top}\widetilde{P}A_K - I = 0$.
Define the Lagrangian:
\[
    \widetilde{\mathcal{L}}(K,\widetilde{P},\widetilde{\Lambda})
    = \log\det\widetilde{P}
      + \tr\!\bigl(\widetilde{\Lambda}^{\top}H(K,\widetilde{P})\bigr)
\]
where $\widetilde{\Lambda}$ is the Lagrange multiplier. 
We first calculate the partial derivative of $   \widetilde{\mathcal{L}}$
w.r.t.\ $\widetilde{P}$: 
$\frac{\partial}{\partial\widetilde{P}}   \widetilde{\mathcal{L}} = \widetilde{P}^{-1}
    + \widetilde{\Lambda} - A_K\widetilde{\Lambda}A_K^{\top}.$
Setting $\partial   \widetilde{\mathcal{L}}/\partial\widetilde{P}=0$ gives 
\begin{equation*}
    \widetilde{P}^{-1}
    + \widetilde{\Lambda} - A_K\widetilde{\Lambda}A_K^{\top} = 0
    \quad\Longrightarrow\quad
    \widetilde{\Lambda} - A_K\widetilde{\Lambda}A_K^{\top}
    = -\widetilde{P}^{-1}.
\end{equation*}
Similar to the proof of Proposition~\ref{prop:am}, we have
\begin{align*}
     &\mathrm{d}\phi
    = \tr\!\bigl(\widetilde{\Lambda}^{\top}\mathrm{d}H
      \big|_{\widetilde{P}}\bigr),   \\
      &\mathrm{d}H\big|_{\widetilde{P}}
    = -\mathrm{d}K^{\top}B^{\top}\widetilde{P}A_K
      - A_K^{\top}\widetilde{P}B\,\mathrm{d}K,\\
      & \tr\!\bigl(\widetilde{\Lambda}^{\top}\mathrm{d}H
      \big|_{\widetilde{P}}\bigr) = -2\,\tr\!\bigl(\mathrm{d}K^{\top}\cdot
      B^{\top}\widetilde{P}A_K\widetilde{\Lambda}\bigr). 
\end{align*}
Comparing with $\mathrm{d}\phi = \tr(\mathrm{d}K^{\top}
\nabla_K\phi)$, the gradient is 
$\nabla_K\phi
    = -2B^{\top}\widetilde{P}A_K\widetilde{\Lambda}$. \textbf{\hspace*{\fill} \qed}

%% file: Sec/Appendix_E.tex
\section*{Appendix E: Proof of Theorem~\ref{thm:conv}}
To formally prove Theorem~\ref{thm:conv}, we need to establish some fundamental properties of $J_\mu$ and $\mathcal{S}_0$.
\subsection*{Appendix E.1: Smoothness of $J_\mu$}
This subsection aims to prove that $J_\mu$ is smooth over $\Dstab$, as shown in the following lemma. 
\begin{lemma}
    The function $J_\mu(K)$ is smooth on
    $K\in\Dstab$.
\end{lemma}
\begin{proof}
    First,  it follows from  the vectorized form of Lyapunov equations that the maps
$K\mapsto P(K)$,  $K\mapsto Z(K)$, $K\mapsto \widetilde{P}(K)$
are smooth  matrix-valued functions on $\Dstab$. And the solutions of Lyapunov equations ensure that $P(K)$, $Z(K)$, and $\widetilde{P}(K)$ are symmetric positive definite. 

Recall that the definition of $J_\mu(K)$ involves   polynomial operations and $\log\det$ over the $P(K)$, $Z(K)$, and $\widetilde{P}(K)$, respectively. Thus, $J_\mu(K)$ is smooth for $K\in\Dstab$.
\hspace*{\fill} \qed
\end{proof}
\subsection*{Appendix E.2: Compactness of the set $\mathcal{S}_0$}
This subsection aims to prove the  compactness of the set  $\mathcal{S}_0 = \{K\in\R^{m\times n} \mid J_\mu(K)\le J_\mu(K_0)\}$, where  $K_0\in\Dstab$ is the initial stabilizing controller  (line 1 in Algorithm~\ref{Alg:Adjoinddescent}). 
We begin by showing that 
$J_\mu(K)$ goes to $+\infty$ for $K\notin \Dstab$, highlighting the barrier role of $\mu\log\det\widetilde{P}(K)$ in the following lemma. 
\begin{lemma}\label{lem:blowup}
Let \(\widetilde P(K)\succ 0\) be the unique solution of
\begin{equation*}
\widetilde P - A_K^\top \widetilde P A_K = I,
\qquad K\in \mathcal{D}_{\mathrm{stab}}.
\end{equation*}
Then, as \(\rho(A_K)\to 1^{-}\),
\begin{equation*}
\log\det \widetilde P(K)\to +\infty, \ J_\mu(K)\to+\infty.
\end{equation*}
\end{lemma}
\begin{proof}
For every \(K\in \mathcal{D}_{\mathrm{stab}}\), the Lyapunov equation admits the convergent series representation \(
\widetilde P
=
\sum_{t=0}^{\infty}(A_K^\top)^t A_K^t
\succeq I. \)
Let \(\lambda\in\mathbb C\) be an eigenvalue of \(A_K\) with \(|\lambda|=\rho(A_K)\), and let \(z\in\mathbb C^n\) be a corresponding unit eigenvector. Then
\begin{align*}
z^* \widetilde P z
&=
\sum_{t=0}^{\infty} z^*(A_K^\top)^t A_K^t z
=
\sum_{t=0}^{\infty} \|A_K^t z\|_2^2 \\
&=
\sum_{t=0}^{\infty} |\lambda|^{2t}
=
\frac{1}{1-\rho(A_K)^2}.
\end{align*}
Since \(\widetilde P\) is symmetric positive definite,
\begin{equation*}
\|\widetilde P\|
=
\lambda_{\max}(\widetilde P)
\ge
z^*\widetilde P z
=
\frac{1}{1-\rho(A_K)^2},
\end{equation*}
and therefore \(\|\widetilde P(K)\|\to+\infty\) as \(\rho(A_K)\to1^{-}\).
Moreover, let $\lambda_i(\widetilde P)$ denote the $i$-th eigenvalue of
$\widetilde P$, ordered as
$\lambda_{\rm min}(\widetilde P)=\lambda_1(\widetilde P)\leq\cdots\leq\lambda_n(\widetilde P)=\lambda_{\rm max}(\widetilde P)$.
Then $\widetilde P\succeq I$ implies $\lambda_i(\widetilde P)\geq 1$ for all $i$, so we have
\begin{equation*}
\begin{aligned}
\log\det \widetilde P
&=
\sum_{i=1}^n \log \lambda_i(\widetilde P)
\ge
\log \lambda_{\max}(\widetilde P) \\
&\ge
-\log\!\bigl(1-\rho(A_K)^2\bigr).
\end{aligned}
\end{equation*}
Hence,  \(\log\det\widetilde P(K)\to+\infty\) as \(\rho(A_K)\to1^{-}\), which   gives  
\[
    J_\mu(K) \ge \mu\log\det\widetilde P(K) \to +\infty. \tag*{$\qed$}
\]
\end{proof}

\begin{lemma}\label{lem:coercivity}
    The function $J_\mu(K)\to+\infty$ as $\|K\|_F\to\infty$. 
\end{lemma}
\begin{proof}
When $K\notin\Dstab$,  it follows from Lemma~\ref{lem:blowup} that $J_\mu(K)=+\infty$. 
 When $K\in \Dstab$,   since $b>0$, $\lmin(R)>0$, $\lmin(\Sigma)>0$, and all other terms are non-negative ($c,d\ge 0$, $P,\Sigma,Z\succeq 0$), we have that when $\|K\|_F\to\infty$, 
\begin{align*}
    J_\mu(K) &\ge J(K)
    \ge b\,\tr(P(K)\,\Sigma) \\&\ge  b\,\tr(Q\Sigma) + b\,\tr(K\Sigma K^{\top}R) \\ 
    &\ge b\,\tr(Q\Sigma) + b\,\lmin(R)\,\lmin(\Sigma)\,\|K\|_F^2
    \to +\infty,
\end{align*}
which completes the proof. \hspace*{\fill} \qed
\end{proof}

We are ready to show the compactness of $\mathcal{S}_0$. 
\begin{lemma}
   The set 
$\mathcal{S}_0$ is compact and $\mathcal{S}_0\subset\Dstab$.
\end{lemma}

\begin{proof}
First, since $K_0$ is stabilizing, $J_\mu(K)\le J_\mu(K_0)<+\infty$ for all $K\in \mathcal{S}_0$. By Lemma~\ref{lem:blowup}, $J_\mu(K)\to+\infty$ as
$\rho(A_K)\to 1^{-}$. Thus, $\mathcal{S}_0\subset\Dstab$
holds and  no sequence in $\mathcal{S}_0$ can
converge to a point on the boundary of $\Dstab$. 

By Lemma~\ref{lem:coercivity}, $J_\mu(K)\to+\infty$ as
$\|K\|_F\to\infty$, so $\mathcal{S}_0$ is bounded.
Since $J_\mu(K)$ is smooth when $K\in \Dstab$  and $J_\mu(K)=+\infty$ when $K\notin \Dstab$, $J_\mu(K)$ is lower semicontinuous for  $K\in \R^{m\times n}$. Thus, 
$\mathcal{S}_0 = \{K : J_\mu(K)\le J_\mu(K_0)\}$ is closed
in $\R^{m\times n}$. According to the Heine-Borel theorem, since $\mathcal{S}_0$ is closed and bounded in $\R^{m\times n}$, one  can conclude that $\mathcal{S}_0$ is compact. \hspace*{\fill} \qed
\end{proof}

\subsection*{Appendix E.3: Backtracking-Armijo line search termination}
This subsection shows that Backtracking-Armijo line search (lines 13--15 of Algorithm~\ref{Alg:Adjoinddescent}) terminates within a finite number of iterations.

\begin{lemma}\label{lem:armijo}
Let  $\beta,\tau\in(0,1)$ and $\eta_0>0$ be the inputs of Algorithm~\ref{Alg:Adjoinddescent}. 
   There exists $\delta_{\min}>0$ such that for every $K\in\mathcal{S}_0$ and
$\eta\in\bigl(0,\delta_{\min}\bigr]$,  
\begin{equation}\label{eq:armijo_ineq}
    J_\mu(K') \;\le\; J_\mu(K)
    - \tau\eta\|\nabla_K J_\mu(K)\|_F^2.
\end{equation}
Furthermore, the Backtracking-Armijo line search (lines 13--15 of Algorithm~\ref{Alg:Adjoinddescent})
terminates in finitely many steps, with accepted step size bounded below by
\begin{equation*}
  \forall k,  \  \eta_k \;\ge\; \eta_{\min},   \ \eta_{\min}
    =\min\!\bigl(\eta_0,\beta\delta_{\min}\bigr)
    > 0. 
\end{equation*}
\end{lemma}
\begin{proof}
Let us first prove the existence of $\delta_{\min}$. 
Since $J_\mu$ is smooth on $\Dstab$ and $\nabla_K J_\mu$ is continuous on compact $\mathcal{S}_0$, there exists a finite $D_{\max}>0$ such that 
\[\|\nabla_K J_\mu(K)\|_F\le D_{\max}, \ \forall K\in \mathcal{S}_0.\] Furthermore, 
since  $\Dstab$ is
open~\citep{bu2020topological} and 
$\mathcal{S}_0\subset\Dstab$ is compact, there exists $\varepsilon>0$ such that the closed tube
$\overline{\mathcal{N}_{\varepsilon/2}}
=\{K:\mathrm{dist}(K,\mathcal{S}_0)\le\varepsilon/2\}
\subset\Dstab$ is compact.
Since $J_\mu$ is smooth in $\Dstab$, the Hessian
$\nabla^2_K J_\mu$ is bounded on
the compact set $\overline{\mathcal{N}_{\varepsilon/2}}$. Define 
 \[L =
    \sup_{K\in\overline{\mathcal{N}_{\varepsilon/2}}}
    \|\nabla^2_K J_\mu(K)\|_{\mathrm{op}} \;<\; \infty\]
where $\|\cdot\|_{\mathrm{op}}$ is the operator norm. 
Set $\delta_1=\varepsilon/(2D_{\max})>0$.
For $\eta\le\delta_1$ and $s\in[0,1]$, the segment point
$K_s= K - s\eta\nabla_K J_\mu(K)$ satisfies:
\[
    \|K_s - K\|_F = s\eta\|\nabla_K J_\mu(K)\|_F
    \le \eta\,D_{\max} \le \varepsilon/2,
\]
so $K_s\in\overline{\mathcal{N}_{\varepsilon/2}}$.
Then, by the Fundamental Theorem of Calculus on the Hilbert space
$(\R^{m\times n},\langle\cdot,\cdot\rangle_F)$, where $\langle\cdot,\cdot\rangle_F$ is the Frobenius inner product:
\begin{align*}
    J_\mu(K')
    &= J_\mu(K) + \int_0^1\langle \nabla_KJ_\mu(K_s),\,
       K'-K\rangle_F\,\mathrm{d}s \\
    &= J_\mu(K) + \langle\nabla_K J_\mu(K),\,K'-K\rangle_F\\
       \\& + \int_0^1\langle\nabla_K J_\mu(K_s)-\nabla_K J_\mu(K),\,
       K'-K\rangle_F\,\mathrm{d}s. 
\end{align*}
 Applying $K'-K = -\eta\nabla_K J_\mu(K)$  gives that 
\begin{align*}
    J_\mu(K')
    &\le J_\mu(K) - \eta\|\nabla_K J_\mu(K)\|_F^2
    + \frac{L\eta^2}{2}\|\nabla_K J_\mu(K)\|_F^2\\
    &= J_\mu(K) - \eta\Bigl(1 - \frac{L\eta}{2}\Bigr)
      \|\nabla_K J_\mu(K)\|_F^2.
\end{align*}
To ensure that the condition~\eqref{eq:armijo_ineq}  holds, we need
$1 - L\eta/2 \ge \tau$, i.e.\ $\eta\le 2(1-\tau)/L$.
Set $\delta_2= 2(1-\tau)/L > 0$.

To sum up, set \(\delta_{\min} =
    \min\!\bigl(\delta_1,\;\delta_2\bigr) \;>\; 0. \) One can conclude that 
for  any $K\in\mathcal{S}_0$ and $\eta\in(0,\delta_{\min}]$, $K'=K - \eta\nabla_K J_\mu(K) \in\Dstab$ and 
          \eqref{eq:armijo_ineq} holds.

Next we prove the finite termination of the backtracking-Armijo line search. The backtracking  sequence is $\eta_0, \beta\eta_0, \beta^2\eta_0,\ldots$
Since $\beta\in (0,1)$, there exists a
\emph{finite} first index $\ell^*$ such that
$\beta^{\ell^*}\eta_0\le \delta_{\min}$ and $\beta^{\ell^*-1}\eta_0> \delta_{\min}$.
At that trial, the condition    \eqref{eq:armijo_ineq} is satisfied, so the
loop exits.  The accepted step is $\eta_k=\beta^{\ell^*}\eta_0$. Then, it is straightforward to show that for all $k$,  $\eta_k = \beta\cdot\beta^{\ell^*-1}\eta_0
    \ge \min\{\eta_0,\beta\cdot\delta_{\min}\}$.  \hspace*{\fill} \qed
\end{proof}

\subsection*{Appendix E.4:  Proof of Theorem~\ref{thm:conv}}
Now we are ready to provide a full proof of Theorem~\ref{thm:conv}. 
Let $\eta_{\min}$ be the same as in Lemma~\ref{lem:armijo}.   
We begin by proving statement (1), i.e., $K_k\in\mathcal{S}_0$, $\forall k$. When $k=0$,  $K_0\in\mathcal{S}_0$ holds by the initialization.  
Suppose $K_k\in\mathcal{S}_0$.
By Lemma~\ref{lem:armijo}, when the backtracking terminates, we have  $J_\mu(K_{k+1})\le J_\mu(K_k)-\tau \eta_k
          \|\nabla_K J_\mu(K_k)\|^2 < J_\mu(K_k)\le J_\mu(K_0)$, which implies that $K_{k+1}\in\mathcal{S}_0$.

The statement (2) directly follows from Lemma~\ref{lem:armijo}, i.e.,  the step sizes satisfy $\eta_k\ge\eta_{\min}>0$
          for all $k$.

To prove the statement (3), we recall that for all $k$, 
\begin{equation*}
    J_\mu(K_{k+1})
    \;\le\; J_\mu(K_k) - \tau\eta_{\min}\|\nabla_K J_\mu(K_k)\|^2.
\end{equation*}
Summing over $k=0,\ldots,N$  gives 
\begin{align}
     \tau \eta_{\min}\sum_{k=0}^{N}\|\nabla_K J_\mu(K_k)\|^2
    &\le J_\mu(K_0) - J_\mu(K_{N+1}) \nonumber\\
   &\hspace{-2cm}\le J_\mu(K_0) - \inf_{K\in \mathcal{S}_0}J_\mu(K)
    < +\infty \label{eq:converge}
\end{align}
which implies that  $\displaystyle\lim_{k\to\infty}\|\nabla_K J_\mu(K_k)\|=0$.

Finally, we prove the sublinear convergence. 
Dividing by $(N+1)$ in \eqref{eq:converge} and bounding the left side below:
\begin{align*}
    \tau\eta_{\min}
    \cdot\min_{0\le k\le N}\|\nabla_K J_\mu(K_k)\|^2 
  &\le
    \frac{1}{N+1}\sum_{k=0}^{N}\|\nabla_K J_\mu(K_k)\|^2 \\
  &\le
    \frac{J_\mu(K_0)-\inf_{K\in \mathcal{S}_0}J_\mu(K)}{N+1}.
\end{align*}
Dividing both sides by $\tau\eta_{\min}>0$ gives~\eqref{eq:rate}. \hspace*{\fill} \qed